\documentclass[10pt]{article}
\usepackage{amsmath,amsthm,amsfonts,amssymb,amscd,overpic,mathrsfs}
\usepackage{dsfont}
\usepackage{amsthm}
\usepackage{amssymb}
\usepackage{graphicx}
\usepackage{hyperref}
\usepackage{enumerate}
\usepackage{xcolor}
\usepackage{amstext,amsthm,amsfonts,amsmath}
\usepackage{enumerate}
\usepackage[utf8]{inputenc}
\usepackage[font=scriptsize]{caption}
\usepackage{indentfirst}
\usepackage{setspace}
\usepackage{cite}

\makeatletter 
\renewcommand\@biblabel[1]{#1.} 
\makeatother

\usepackage{titlesec}
\titleformat{\subsection}
  {\normalfont\fontfamily\selectfont}{\thesubsection}{1em}{}

\newcommand{\norm}[1]{\left\Vert#1\right\Vert}
\newcommand{\set}[1]{\left\{#1\right\}}
\newcommand{\exset}[2]{S_{e}\left(#1,#2\right)}
\newcommand{\dist}[2]{\textnormal{dist}\left(#1,#2\right)}
\newcommand{\inpr}[2]{\left\langle#1,#2\right\rangle}
\newcommand{\graf}[1]{\textnormal{Gr}\left(#1\right)}
\newcommand{\R}{\mathbb{R}}

\newtheorem{algorithm}{Algorithm}
\newtheorem*{algorithmn}{Algorithm}
\newtheorem{theorem}{Theorem}[section]
\newtheorem{lemma}{Lemma}[section]
\newtheorem{proposition}{Proposition}[section]
\newtheorem{definition}{Definition}[section]
\newtheorem{corollary}{Corollary}[section]

\title{On the Complexity of the Projective Splitting and Spingarn's Methods for the
Sum of Two Maximal Monotone Operators}
\author{Majela Pent\'on Machado
\thanks{IMPA, Estr.\,Dona Castorina 110, 22460-320 Rio de Janeiro, Brazil,
       \href{mailto:majela@impa.br}{\tt majela@impa.br}.}
}
\date{}
\begin{document}

\maketitle

\abstract{In this work, we study the pointwise and ergodic iteration-complexity of a family of projective splitting methods proposed by Eckstein and Svaiter, for finding a zero of the sum of two maximal monotone operators. As a consequence of the complexity analysis of the projective splitting methods, we obtain complexity bounds for the two-operator case of Spingarn's partial inverse method. We also present inexact variants of two specific instances of this family of algorithms and derive corresponding convergence rate results.
}

\medskip


\section{Introduction}
\label{sec:introduction}
A wide variety of problems, such as optimization and min-max problems, 
complementarity problems and variational inequalities, can be posed as the \emph{monotone inclusion problem} (MIP) associated to a maximal monotone point-to-set operator.
An important tool for the design and analysis of several implementable methods for solving MIPs is the
\emph{proximal point algorithm} (PPA), proposed by Martinet \cite{martinet_1970} and generalized by Rockafellar 
\cite{rock_ppa}. 
Even though the PPA has good global and local convergence properties \cite{rock_ppa},
its major drawback is that it requires the evaluation of the \emph{resolvent mappings} (or \emph{proximal mappings}) associated with the operator. The difficulty lies in the fact that evaluating a resolvent mapping, which is equivalent to solving a \emph{proximal
subproblem}, can be as complicated as finding a root of the operator.
 
One alternative to surmount this difficulty is to decompose the operator as the sum of two maximal monotone operators such that their resolvents are considerably easier to evaluate. Then, one can devise methods that use independently these proximal mappings.

In this work, we are concerned with MIPs defined by the sum of two maximal monotone
operators. 
We are also interested in the case 
where the problems of finding zeros of these operators separately are easier than solving the MIP for the sum. 
A typical instance of this situation is the \emph{variational inequality problem} associated with a maximal monotone operator $A$ and a closed convex subset $C$, whose solutions are precisely the zeros of the sum of $A$ and the normal operator associated with $C$, known to be maximal monotone.

\emph{Splitting methods} (or \emph{decomposition methods}) for problems of the above-mentioned type attempt to converge to
a solution of the MIP by solving, in each iteration, subproblems involving one of the operators, but not both.
\emph{Peaceman-Rachford} and \emph{Douglas-Rachford} methods are examples of this type of algorithms. These 
were first introduced in \cite{peaceman_rachford} and \cite{douglas_rachford} for the particular case of 
linear mappings, and then generalized in \cite{lions_mercier} by Lions and Mercier to address MIPs.
\emph{Forward-backward} methods \cite{lions_mercier,passty_1979_FB,tseng_2000_modified}, 
which generalize standard gradient projection methods for variational inequalities and optimization 
problems, are also examples of splitting algorithms.

Recently, a new family of splitting methods for solving MIPs given by the sum of two maximal monotone operators was introduced in \cite{eck_sv_08} by Eckstein and Svaiter. Through a generalized solution set in a product space, whose projection onto the first space is indeed the solution set of the problem, the authors constructed a class of decomposition methods with quite solid convergence properties.
These algorithms are
essentially projection methods, in the sense that in each
iteration a hyperplane is constructed separating the current iterate from the generalized solution
set, and then the next
iterate is taken as a relaxed projection of the current one onto this separating hyperplane. In order
to construct such hyperplanes, two proximal subproblems are solved,
each of which involves only one of the two maximal monotone operators,
which ensures the splitting nature of the methods.

In this work we study the iteration-complexity of the family of methods proposed
in \cite{eck_sv_08}, to be referred as \emph{projective splitting methods} (PSM) in the sequel. 
We start our analysis by introducing a projective algorithm that generalizes the PSM.
We then consider a termination criterion for this general algorithm in terms of the 
$\epsilon$-enlargements of the operators, which allows us to obtain convergence rates for the 
PSM measured by the pointwise and ergodic iteration-complexities.

Using the complexity analysis developed for the PSM, we also study the complexity 
of Spingarn's splitting method 
for solving inclusion problems given by the sum of two maximal monotone operators. 
In \cite{spingarn}, Spingarn introduced a splitting method for finding a zero of the sum of $m$ maximal monotone 
operators using the concept of \emph{partial inverses}. For the two-operator case, Eckstein and Svaiter proved 
in \cite{eck_sv_08} that Spingarn's method is a special case of a scaled variant of the PSM. This will allow us to establish iteration-complexity results for Spingarn's method for the case of the sum of two maximal monotone operators.

The general projective method that we introduce in this work is also used to construct inexact variants of two special cases of the family of PSM. For two specific instances of the PSM, we consider a relative error condition for 
approximately evaluating the resolvents. The error criterion considered in this work is different from the one 
used in \cite{eck_sv_09}, where was generalized the projective splitting framework for MIPs given by the sum of $m$ maximal 
monotone operators. Indeed, we will use the notion of approximate solutions of a proximal subproblem presented in 
\cite{sol_sv_unif}, which yields a more flexible error tolerance criterion and allows evaluation of the $\epsilon$-enlargement. 
We also derive convergence rate results for these two novel algorithms. 

The remainder of this paper is organized as follows. Section \ref{preliminares} reviews the definitions and some basic properties of a point-to-set maximal monotone operator and its $\epsilon$-enlargements. Section \ref{sec:general framework} presents a relaxed
projection method that extends the framework introduced in \cite{eck_sv_08}. It also proves some properties regarding this general scheme and establishes the stopping criterion that will be considered for such method and its instances.
Section \ref{complexity} presents the PSM introduced in \cite{eck_sv_08} and derives global convergence rate results for these methods. Subsection \ref{specialized} specializes these general complexity bounds for the case where global convergence for the family of PSM was obtained in \cite{eck_sv_08}. Section \ref{sec:spingarn} studies the  iteration complexity of the two-operator case of Spingarn’s method of partial inverses \cite{spingarn}. Finally, sections \ref{inexact} and \ref{sequential} propose inexact versions of two special cases of the PSM and establish iteration-complexity results for them.

\section{Preliminaries}
\label{preliminares}

Throughout this paper, we let $\R^n$ denote an $n$-dimensional space with inner product and induced norm denoted by $\inpr{\cdot}{\cdot}$ and $\norm{\cdot}$, respectively. We also define the spaces $\R_{+}$ and $\mathbb{E}$ as $\R_{+}:=\set{x\in\R\,:\,x\geq0}$ and $\mathbb{E}:=\R^n\times\R^n\times\R_{+}$.

In what follows in this section, we will review some material related to a point-to-set maximal monotone operator and its $\epsilon$-enlargements that will be needed along this work. 

A point-to-set operator $T:\mathbb{R}^{n}\rightrightarrows\mathbb{R}^{n}$ is a relation $T \subseteq\R^n\times\R^n$ and 
\begin{equation*}
T(z) := \{v\in\mathbb{R}^{n}:(z,v)\in T\}\qquad z\in \mathbb{R}^n.
\end{equation*}
Given $T:\mathbb{R}^{n}\rightrightarrows\mathbb{R}^{n}$ its graph is the set
\begin{equation*}
\graf{T}:=\{(z,v)\in\R^n\times\R^n:v\in T(z)\}.
\end{equation*}
An operator $T:\mathbb{R}^{n}\rightrightarrows\mathbb{R}^{n}$ is \emph{monotone}, if 
\begin{equation*}
\inpr{z-z'}{v-v'}\geq0\qquad\qquad\forall (z,v),(z',v')\in\graf{T},
\end{equation*}
and it is \emph{maximal monotone} if it is monotone and maximal in the family of monotone operators of $\mathbb{R}^{n}$ into $\mathbb{R}^{n}$, with respect to the partial order of inclusion. This is, if $S:\mathbb{R}^{n}\rightrightarrows\mathbb{R}^{n}$ is a monotone operator such that $\graf{T}\subseteq\graf{S}$, then $S=T$.

The \emph{resolvent mapping} of a maximal monotone operator $T$ with parameter $\lambda>0$ is $(I+\lambda T)^{-1}$, where $I$ is the identity mapping. It follows directly from the definition that $z'=(I+\lambda T)^{-1}(z)$, if and only if $z'$ is the solution of the \emph{proximal subproblem} 
\begin{equation}
\label{eq:prox-sub}
0\in\lambda T(z') + (z'-z).
\end{equation}

The $\epsilon$-\emph{enlargement} of a maximal monotone operator was introduced in \cite{bur_ius_sv_97} by Burachik, Iusem and Svaiter. In \cite{mon_sv_hpe}, Monteiro and Svaiter extended this notion to a generic point-to-set operator as follows. Given $T:\mathbb{R}^{n}\rightrightarrows\mathbb{R}^{n}$ and $\epsilon\in\mathbb{R}$, define the operator $\epsilon$-enlargement of $T$, $T^{\epsilon}:\mathbb{R}^{n}\rightrightarrows\mathbb{R}^{n}$, by
\begin{equation*}
T^{\epsilon}(z) := \{v\in\mathbb{R}^{n}:\inpr{z'-z}{v'-v}\geq-\epsilon,\quad\forall (z',v')\,\in\,\graf{T}\},\qquad\forall z\in\mathbb{R}^{n}.
\end{equation*}
The following proposition presents some important properties of $T^{\epsilon}$. Its proof can be found in \cite{mon_sv_hpe}.
\begin{proposition}
\label{Prop enlargement}
Let $T:\mathbb{R}^{n}\rightrightarrows\mathbb{R}^{n}$. Then,
\begin{itemize}
\item[(a)] if $\epsilon'\leq\epsilon$, we have $T^{\epsilon'}(z)\subseteq T^{\epsilon}(z)$ for all $z\in\mathbb{R}^{n}$;
\item[(b)] $T$ is monotone if and only if $T\subseteq T^{0}$;
\item[(c)] $T$ is maximal monotone if and only if $T=T^{0}$.
\end{itemize}
\end{proposition}

Observe that items (a) and (c) above imply that, if $T:\R^n\rightrightarrows\R^n$ is maximal monotone, then $T(z)\subseteq T^\epsilon(z) $ for all $z\in\R^n$ and $\epsilon\geq0$.
Hence, $T^\epsilon(z)$ is indeed an enlargement of $T(z)$.

We now state the \emph{weak transportation formula} \cite{bur_sag_sv_99} for computing points in the graph of $T^\epsilon$. This formula will be used in the complexity analysis of some ergodic iterates generated by the algorithms studied in this work (see subsection \ref{ergodic iterates}).
\begin{theorem}
\label{Teo Transp For}
Assume that $T:\mathbb{R}^{n}\rightrightarrows\mathbb{R}^{n}$ is a maximal monotone operator. Let $z_{i},v_{i}\in\mathbb{R}^{n}$ and $\epsilon_{i},\alpha_{i}\in\mathbb{R}_{+}$, for $i=1,\ldots,k$, be such that
\begin{align*}
v_{i}\in T^{\epsilon_{i}}(z_{i}),\quad i=1,\ldots,k, \qquad\qquad\sum_{i=1}^{k}\alpha_{i}=1,
\end{align*}
and define 
\begin{equation*}
\overline{z}: = \sum_{i=1}^{k}\alpha_{i}z_{i}, \qquad\quad \overline{v}:=\sum_{i=1}^{k}\alpha_{i}v_{i}, \qquad \quad\overline{\epsilon}:=\sum_{i=1}^{k}\alpha_{i}(\epsilon_{i}+\inpr{z_{i}-\overline{z}}{v_{i}}).
\end{equation*}
Then, $\overline{\epsilon}\geq0$ and $\overline{v}\in T^{\overline{\epsilon}}(\overline{z})$.
\end{theorem}

\section{The General Projective Splitting Framework}
\label{sec:general framework}
The monotone inclusion problem (MIP) of interest in this work consists of finding
$z\in\mathbb{R}^{n}$ such that
\begin{equation}
\label{problem}
0 \in A(z) + B(z),
\end{equation}
where $A,B:\mathbb{R}^{n}\rightrightarrows\mathbb{R}^{n}$ are maximal
monotone operators. 

The framework presented in \cite{eck_sv_08} reformulates problem \eqref{problem} in  terms  of  a  convex  
feasibility problem, which  is  defined  by  a certain closed convex \emph{extended solution set}.
To solve the feasibility problem, the authors introduced successive projection 
algorithms that use, at each iteration, independent calculations involving each operator. Our goals in this section are to present a scheme that generalizes the methods in \cite{eck_sv_08}, and to study its properties. This general framework will allow us to derive convergence rates for the family of PSM and Spingarn's method. In addition, using this general method, we construct inexact versions of two specific instances of the PSM and study their complexities.

Consider $S_e(A,B)\subset\R^n\times\R^n$ the \emph{extended solution set} of (\ref{problem}) defined in
\cite{eck_sv_08} as
\begin{equation*}
S_e(A,B):=\set{(z,w)\in\R^n\times\R^n\,:\,w\in B(z),\, -w\in A(z)}.
\end{equation*}
The following result establishes two important properties of $\exset{A}{B}$. Its proof can be found in \cite[Lemma 1]{eck_sv_08}.

\begin{lemma}
\label{lem:ext-set}
  If $A,B:\mathbb{R}^{n}\rightrightarrows\mathbb{R}^{n}$ are maximal
  monotone operators, then the following statements hold.
\begin{itemize}
\item[(a)] A point $z\in\mathbb{R}^{n}$ is a solution of \eqref{problem},
  if and only if there is $w\in\mathbb{R}^{n}$ such that
  $(z,w)\in\exset{A}{B}$.
\item[(b)] $S_e(A,B)$ is a closed and convex subset of $\R^n\times\R^n$.
\end{itemize}
\end{lemma}

According to the above lemma, problem \eqref{problem} is equivalent to the 
convex feasibility problem of finding a point in $S_e(A,B)$.
In order to solve this feasibility problem by successive orthogonal projection methods, we need to construct hyperplanes separating points $(z,w)\notin S_e(A,B)$ from $S_e(A,B)$.
For this purpose, in \cite{eck_sv_08} it was used points in the graph of $A$ and $B$ to define affine functions, which were called \emph{decomposable separators}, such that
$\exset{A}{B}$ was contained in the non-positive half-spaces determined by
them. Here, we generalize this concept using points in
the $\epsilon$-enlargements of $A$ and $B$.

\begin{definition}
  \label{def_dec_sep_gen}
\emph{  Given two triplets $(x,b,\epsilon^{x})$,
  $(y,a,\epsilon^{y})\in\mathbb{E}$ such that $b\in B^{\epsilon^{x}}(x)$
  and $a\in A^{\epsilon^{y}}(y)$, the\,\, \emph{decomposable\,\, separator}\,\, associated\,\, with\,\, $(x,b,\epsilon^{x})$\,\, and\,\, $(y,a,\epsilon^{y})$\,\, is\,\, the\,\, affine\,\, function $\phi:\R^n\times\R^n\rightarrow\mathbb{R}$
  \begin{equation*}
    \phi(z,w) := \inpr{z-x}{b-w} + \inpr{z-y}{a+w} -\epsilon^{x} - \epsilon^{y}.
  \end{equation*}
  The \emph{non-positive level set} of $\phi$ is
  \begin{align*}
      H_\phi:=\set{(z,w)\in\R^n\times\R^n\;:\;\phi(z,w) \leq 0}.
  \end{align*}}
\end{definition}

\begin{lemma}
\label{lema phi propt}
  If $\phi$ is the decomposable separator associated with $(x,b,\epsilon^{x})$ and
  $(y,a,\epsilon^{y})\in\mathbb{E}$, where $b\in B^{\epsilon^{x}}(x)$
  and $a\in A^{\epsilon^{y}}(y)$, and $H_\phi$ is its non-positive level set, then
  \begin{itemize}
  \item[(a)] $ S_e(A,B)\subseteq H_\phi$;
  \item[(b)] either $\nabla \phi\neq 0$ or $\phi\leq 0$ in $\mathbb{R}^n\times
    \mathbb{R}^n$;
  \item[(c)] either $H_\phi$ is a closed half-space or $H_\phi=\R^n\times\R^n$.
  \end{itemize}
\end{lemma}

\begin{proof}
Item (a) is a direct consequence of the definitions of the $\epsilon$-enlargement of a point-to-set
operator and the set $S_e(A,B)$. Rewriting $\phi(z,w)$  as
\begin{equation}
\label{rewriting_phi}
\phi(z,w) = \inpr{z-y}{a+b} + \inpr{w-b}{x-y} -\epsilon^x -\epsilon^y\qquad \forall(z,w)\in\R^n\times\R^n,
\end{equation}
and noting that $\nabla\phi=(a+b,x-y)$ and $\epsilon^x$, $\epsilon^y\geq0$, then (b) and (c) follow immediately.
\end{proof}

We now present the general projection scheme for finding a point in $S_e(A,B)$ that will be studied in this work. 
Algorithm \ref{alg_genr_proj} below generalizes the framework introduced in \cite{eck_sv_08}, since we use the notion of decomposable separator introduced in Definition \ref{def_dec_sep_gen}. 

\begin{algorithm}
\label{alg_genr_proj}
Choose $(z_{0},w_{0})\in\R^n\times\R^n$. For $k=1,2,\ldots$
\begin{itemize}
\item[1.] Choose $(x_{k},b_{k},\epsilon^x_{k})$ and
  $(y_{k},a_{k},\epsilon^y_{k})\in\mathbb{E}$ such that
\begin{align*}
  b_{k}\in B^{\epsilon^x_{k}}(x_{k})\hspace{5mm}\text{and} \hspace{5mm}a_{k}\in A^{\epsilon^y_{k}}(y_{k}).
\end{align*}
\item[2.] Define $\phi_{k}:\R^n\times\R^n\rightarrow\mathbb{R}$ as the
  decomposable separator associated with $(x_{k},b_{k},\epsilon^x_{k})$
  and $(y_{k},a_{k},\epsilon^y_{k})$, and compute $P_{H_{\phi_k}}(z_{k-1},w_{k-1})$, the orthogonal projection of $(z_{k-1},w_{k-1})$ onto the set $H_{\phi_{k}}$, i.e. define
  \begin{equation*}
 \begin{split}
  & \gamma_k:=
   \begin{cases}
     0,&\text{if }\phi_k(z_{k-1},w_{k-1})\leq 0,\\
     \dfrac{\phi_k(z_{k-1},w_{k-1})}{\norm{\nabla\phi_k}^2},&\text{otherwise};
   \end{cases}\\
   \end{split}
   \end{equation*}
   and set
   \begin{equation*}
   P_{H_\phi}(z_{k-1},w_{k-1})=(z_{k-1},w_{k-1})-\gamma_k\nabla\phi_{k}.
 \end{equation*}
\item[3.] Choose $\rho_{k}\in\left]0,2\right[$ and set
\begin{equation*}
\begin{split}
(z_{k},w_{k}) & = (z_{k-1},w_{k-1}) + \rho_{k}\left[P_{H_{\phi_k}}(z_{k-1},w_{k-1})-(z_{k-1},w_{k-1})\right]\\
& = (z_{k-1},w_{k-1}) - \rho_{k}\gamma_{k}\nabla\phi_{k}.
\end{split}
\end{equation*}
\end{itemize}
\end{algorithm}

Note that the general form of Algorithm \ref{alg_genr_proj} is not sufficient to guarantee convergence of the sequence $\{(z_{k},w_{k})\}$ to a point in $\exset{A}{B}$. 
For example, if the separation between the point $(z_{k-1},w_{k-1})\notin S_e(A,B)$ and $S_e(A,B)$ by $\phi_{k}$ is not strict, then the next iterate is in fact $(z_{k-1},w_{k-1})$ itself, which
might lead to a constant sequence. 
Hence, to ensure convergence it is necessary to impose
additional conditions on the decomposable separators, see \cite{eck_sv_08} and sections \ref{complexity}, \ref{inexact} and \ref{sequential} below. However, since Algorithm \ref{alg_genr_proj} is a relaxed projection type method, it is possible to establish Fejér 
monotone convergence to $S_e(A,B)$ and boundedness of its generated sequence, as well as other classical properties of this kind of algorithms (see for example \cite{eck_sv_08}, \cite{bau_borw_96}).

\subsection{The Generated Sequences}
\label{background}

We will now analyze some properties of the sequences $\{(z_{k},w_{k})\}$, $\{\phi_k\}$, $\{\gamma_k\}$ and $\{\rho_k\}$ generated by Algorithm \ref{alg_genr_proj}, which will be needed in our complexity study.
To this end, let us first prove the following technical result.

\begin{lemma}
\label{Prop Phi}
For any $(z,w)\in\R^n\times\R^n$ and $k\geq1$ we have
\begin{equation}
\label{sum identity Phi}
\dfrac12\norm{(z,w)-(z_k,w_k)}^2 +\, \dfrac12\sum_{j=1}^k\rho_j(2-\rho_j)\gamma_j^2\norm{\nabla\phi_j}^2 = \dfrac12\norm{(z,w)-(z_0,w_0)}^2 +\, \sum_{j=1}^k\rho_j\gamma_j\phi_j(z,w).
\end{equation}
\end{lemma}
\begin{proof}
First we observe that for $j=1,2,\ldots,$ and any $(z,w)\in\R^n\times\R^n$ it holds that
\begin{equation}
\label{eq:eq1-prop-phi}
\begin{split}
\dfrac{1}{2} \norm{(z,w)-(z_j,w_j)}^2  = &\, \dfrac{1}{2}\norm{(z,w)-(z_{j-1},w_{j-1})+\rho_j\gamma_j\nabla\phi_j}^2\\
 = &\, \dfrac{1}{2}\norm{(z,w)-(z_{j-1},w_{j-1})}^2 + \inpr{(z,w)-(z_{j-1},w_{j-1})}{\rho_j\gamma_j\nabla\phi_j}\\
& \,\, + \dfrac{1}{2}\rho_j^2\gamma_j^2\norm{\nabla\phi_j}^2\\
= & \, \dfrac{1}{2}\norm{(z,w)-(z_{j-1},w_{j-1})}^2 + \rho_j\gamma_j\inpr{(z,w)-(y_j,b_j)}{\nabla\phi_j}\\
& \,\, + \rho_j\gamma_j\inpr{(y_j,b_j)-(z_{j-1},w_{j-1})}{\nabla\phi_j} + \dfrac{1}{2}\rho_j^2\gamma_j^2\norm{\nabla\phi_j}^2,
\end{split}
\end{equation}
where the first equality above follows from the update rule in step $3$ of Algorithm \ref{alg_genr_proj}.

Equation \eqref{rewriting_phi} with $\phi=\phi_j$ implies that
\begin{equation*}
\phi_j(z,w) = \inpr{(z,w)-(y_j,b_j)}{\nabla\phi_j} - \epsilon^x_j - \epsilon^y_j\quad\qquad \forall(z,w) \in \R^n\times\R^n.
\end{equation*}
Therefore, adding and subtracting $\rho_j\gamma_j(\epsilon^x_j + \epsilon^y_j)$ on the right-hand side of the last equality in \eqref{eq:eq1-prop-phi} and combining with the identity above, we obtain 
\begin{equation*}
\begin{split}
\dfrac{1}{2} \norm{(z,w)-(z_j,w_j)}^2 = & \,\dfrac{1}{2}\norm{(z,w)-(z_{j-1},w_{j-1})}^2 + \rho_j\gamma_j\phi_j(z,w)\\
& \,\,\,\, - \rho_j\gamma_j\phi_j(z_{j-1},w_{j-1}) + \dfrac{1}{2}\rho_j^2\gamma_j^2\norm{\nabla\phi_j}^2.
\end{split}
\end{equation*}

If we assume that $\gamma_j>0$, then the definition of $\gamma_j$ in step 2 of Algorithm \ref{alg_genr_proj} yields that $\phi_j(z_{j-1},w_{j-1})=\gamma_j\norm{\nabla\phi_j}^2$. Hence, substituting this expression into the equality above and rearranging, we have
\begin{equation*}
\dfrac{1}{2} \norm{(z,w)-(z_j,w_j)}^2 + \dfrac{1}{2}\rho_j(2-\rho_j)\gamma_j^2\norm{\nabla\phi_j}^2 =  \,\dfrac{1}{2}\norm{(z,w)-(z_{j-1},w_{j-1})}^2 + \rho_j\gamma_j\phi_j(z,w).
\end{equation*}
It is clear that this latter equality also holds if $\gamma_j=0$. Thus, adding equation above from $j=1$ to $k$ we obtain \eqref{sum identity Phi}.
\end{proof}

In what follows we assume that problem \eqref{problem} has at least one solution, which implies that $S_e(A,B)$ is a non-empty set in view of Lemma \ref{lem:ext-set}.

Next theorem, which follows directly from Lemma \ref{Prop Phi}, establishes boundedness of the sequence $\{(z_k,w_k)\}$ calculated by Algorithm \ref{alg_genr_proj}, and it also shows that the sum appearing on the left-hand 
side of \eqref{sum identity Phi} is bounded by the distance of the initial point to the set $S_e(A,B)$.

\begin{theorem}
\label{Prop general proj}
Take  $(z_{0},w_{0})\in\R^n\times\R^n$ and let 
$\{(z_{k},w_{k})\}$, $\{\phi_k\}$, $\{\gamma_k\}$ and $\{\rho_k\}$
be the sequences generated by Algorithm $\ref{alg_genr_proj}$. Then, for every integer $k\geq1$, we have
\begin{align}
\label{dist_est}
\sum_{j=1}^{k}\rho_{j}(2-\rho_{j})\gamma_{j}^{2}\norm{\nabla\phi_{j}}^{2} \leq d_{0}^{2}\qquad\text{and}\qquad\norm{(z_{k},w_{k})-(z_0,w_0)} \leq 2d_{0},
\end{align}
where $d_{0}$ is the distance of $(z_{0},w_{0})$ to $\exset{A}{B}$.
\end{theorem}

\begin{proof}
Take $(z^\ast,w^\ast)$ the orthogonal projection of $(z_0,w_0)$ onto $S_e(A,B)$. From Lemma \ref{lema phi propt}(a) it follows that $\phi_j(z^\ast,w^\ast)\leq0$ for all integer $j\geq1$. 
Hence, specializing equality \eqref{sum identity Phi} with $(z^\ast,w^\ast)$ we obtain the first bound in \eqref{dist_est} and the following inequality
\begin{equation*}
\norm{(z_k,w_k)-(z^\ast,w^\ast)}\leq d_0.
\end{equation*}
Since $\norm{(z_0,w_0)-(z^\ast,w^\ast)}=d_0$, the second estimate in \eqref{dist_est} follows from the latter two relations and the triangle inequality for norms.
\end{proof}

It is important to say that Theorem \ref{Prop general proj} can be proven using standard arguments of relaxed projection algorithms, see for instance \cite{bau_borw_96}. 
We have chosen the above approach since it will be more convenient for our subsequent analysis.

\subsection{The Ergodic Sequences}
\label{ergodic iterates}

Besides the pointwise complexity of specific instances of Algorithm \ref{alg_genr_proj},
we are also interested in deriving their convergence rates measured by the iteration 
complexity in an ergodic sense. To do this, we consider sequences obtained by weighted 
averages of the sequences $\{x_{k}\}$ and $\{y_{k}\}$, generated by Algorithm \ref{alg_genr_proj}, and study their properties.

Let $\{x_{k}\}$, $\{y_{k}\}$, $\{\gamma_k\}$ and $\{\rho_k\}$ be the sequences computed with Algorithm \ref{alg_genr_proj}, for every integer $k\geq1$ assume that $\gamma_k>0$ and 
define $\overline{x}_{k}$ and $\overline{y}_{k}$ as                                                                              
\begin{equation}
\label{def_xy_erg}
\overline{x}_{k} := \frac{1}{\Gamma_{k}}\sum_{j=1}^{k}\rho_{j}\gamma_{j}x_{j},\qquad \overline{y}_{k} := \frac{1}{\Gamma_{k}}\sum_{j=1}^{k}\rho_{j}\gamma_{j}y_{j},\qquad\quad \textrm{where} \hspace{3mm} \Gamma_{k} := \sum_{j=1}^{k}\rho_{j}\gamma_{j}.
\end{equation}
The following lemma is a direct consequence of the weak transportation formula, Theorem \ref{Teo Transp For}.

\begin{lemma}
\label{Lem_residual_erg}
Let $\{(x_{k},b_{k},\epsilon^x_{k})\}$, $\{(y_{k},a_{k},\epsilon^y_{k})\}$, $\{\gamma_k\}$ and $\{\rho_k\}$ be the sequences generated by Algorithm $\ref{alg_genr_proj}$. 
For every integer $k\geq1$, suppose that $\gamma_k>0$ and consider $\overline{x}_{k}$, $\overline{y}_{k}$ and $\Gamma_k$ given as in \eqref{def_xy_erg}. Define also
\begin{align}
\label{def_res_erg_x}
\overline{b}_{k} := \frac{1}{\Gamma_{k}}\sum_{j=1}^{k}\rho_{j}\gamma_{j}b_{j}, & & \overline{\epsilon}^x_{k} := \frac{1}{\Gamma_{k}}\sum_{j=1}^{k}\rho_{j}\gamma_{j}(\epsilon^x_{j}+\inpr{x_{j}-\overline{x}_{k}}{b_{j}}),\\
\label{def_res_erg_y}
\overline{a}_{k} := \frac{1}{\Gamma_{k}}\sum_{j=1}^{k}\rho_{j}\gamma_{j}a_{j}, & & \overline{\epsilon}^y_{k} := \frac{1}{\Gamma_{k}}\sum_{j=1}^{k}\rho_{j}\gamma_{j}(\epsilon^y_{j}+\inpr{y_{j}-\overline{y}_{k}}{a_{j}}).
\end{align}
Then, we have
\begin{align*}
\overline{\epsilon}^x_{k}\geq0,\hspace{17mm} & \overline{b}_{k}\in B^{\overline{\epsilon}^x_{k}}(\overline{x}_{k}),\\
\overline{\epsilon}^y_{k}\geq0,\hspace{17mm} & \overline{a}_{k}\in A^{\overline{\epsilon}^y_{k}}(\overline{y}_{k}).
\end{align*}
\end{lemma}

\begin{proof}
The lemma follows from Theorem \ref{Teo Transp For} and the inclusions $b_{j}\in B^{\epsilon^x_{j}}(x_{j})$ and $a_{j}\in A^{\epsilon^y_{j}}(y_{j})$.
\end{proof}

We will refer to the sequences $\{(\overline{x}_k,\overline{b}_k,\overline{\epsilon}^x_{k})\}$ and $\{(\overline{y}_k,\overline{a}_k,\overline{\epsilon}^y_{k})\}$, defined in \eqref{def_xy_erg}-\eqref{def_res_erg_y}, as the \emph{ergodic sequences} associated with Algorithm \ref{alg_genr_proj}.

Next lemma presents alternative expressions for $\overline{a}_k+\overline{b}_k$, $\overline{x}_k-\overline{y}_k$ and
$\overline{\epsilon}^x_{k}+\overline{\epsilon}^y_{k}$, which will be used for obtaining bounds on their sizes.
\begin{lemma}
\label{Prop rewriten residuals}
Let $\{(x_{k},b_{k},\epsilon^x_{k})\}$, $\{(y_{k},a_{k},\epsilon^y_{k})\}$, $\{\gamma_k\}$ and $\{\rho_k\}$ be the sequences generated by Algorithm $\ref{alg_genr_proj}$. Assume that $\gamma_k>0$ for all $k\geq1$, and define the sequences $\{\overline{x}_{k}\}$, $\{\overline{y}_{k}\}$, $\{\Gamma_k\}$, $\{\overline{b}_k\}$, $\{\overline{a}_k\}$, $\{\overline{\epsilon}^x_{k}\}$ and $\{\overline{\epsilon}^y_{k}\}$ as in \eqref{def_xy_erg}, \eqref{def_res_erg_x} and \eqref{def_res_erg_y}. Then, for every integer $k\geq1$, we have
\begin{equation}
\label{eq:rew-erg-ite}
\overline{a}_{k}+\overline{b}_{k} = \frac{1}{\Gamma_{k}}(z_{0}-z_{k}),
\qquad \overline{x}_{k}-\overline{y}_{k} = \frac{1}{\Gamma_{k}}(w_{0}-w_{k}),
\qquad \overline{\epsilon}^x_{k} + \overline{\epsilon}^y_{k} = -\frac{1}{\Gamma_k}\sum_{j=1}^k\rho_j\gamma_j\phi_j(\overline{y}_k,\overline{b}_k).
\end{equation}
\end{lemma}

\begin{proof}
Direct use of the definitions of\, $\overline{x}_k$, $\overline{y}_k$, $\overline{b}_k$ and $\overline{a}_k$ yields
\begin{equation*}
(\overline{a}_{k}+\overline{b}_{k},\overline{x}_{k}-\overline{y}_{k}) = \frac{1}{\Gamma_{k}}\sum_{j=1}^{k}\rho_{j}\gamma_{j}(a_{j}+b_{j},x_{j}-y_{j}).
\end{equation*}
Since $\nabla\phi_j=(a_j+b_j,x_j-y_j)$ for all integer $j\geq1$, in view of the update rule in step 3 of Algorithm \ref{alg_genr_proj}, the definition of $\Gamma_k$ and the equation above, we have
\begin{equation*}
\begin{split}
(z_{k},w_{k}) = & \,(z_{0},w_{0}) - \sum_{j=1}^{k}\rho_{j}\gamma_{j}(a_{j}+b_{j},x_{j}-y_{j})\\
 = &\, (z_{0},w_{0}) - \Gamma_{k}(\overline{a}_{k}+\overline{b}_{k},\overline{x}_{k}-\overline{y}_{k}).
\end{split}
\end{equation*}
The relation above clearly implies the first two identities in \eqref{eq:rew-erg-ite}.

To prove the last equality in \eqref{eq:rew-erg-ite} we first note that
\begin{align*}
\sum_{j=1}^k\rho_j\gamma_j\phi_j(\overline{y}_k,\overline{b}_k)& = \sum_{j=1}^k \rho_j\gamma_j\left(\inpr{\overline{y}_k-x_j}{b_j-\overline{b}_k}+\inpr{\overline{y}_k-y_j}{a_j+\overline{b}_k} -\epsilon^x_{j} -\epsilon^y_{j}\right)\\
&= \sum_{j=1}^k \rho_j\gamma_j\left(\inpr{\overline{y}_k}{b_j} + \inpr{x_j}{\overline{b}_k-b_j} +\inpr{\overline{y}_k-y_j}{a_j} -\inpr{y_j}{\overline{b}_k} -\epsilon^x_{j} -\epsilon^y_{j}\right).
\end{align*}
Next, we multiply the equation above by $1/\Gamma_k$ and use the definitions of $\overline{y}_{k}$ and $\overline{b}_{k}$ to obtain
\begin{align}
\label{Phi(y,b)}
\frac{1}{\Gamma_k}\sum_{j=1}^k\rho_j\gamma_j\phi_j(\overline{y}_k,\overline{b}_k)& = \frac{1}{\Gamma_k}\sum_{j=1}^k \rho_j\gamma_j\left(\inpr{x_j}{\overline{b}_k-b_j}+\inpr{\overline{y}_k-y_j}{a_j} -\epsilon^x_{j} -\epsilon^y_{j}\right).
\end{align}
Now, we observe that $\overline{\epsilon}^x_{k}$ can be rewritten as
\begin{align*}
\overline{\epsilon}^x_{k} = \frac{1}{\Gamma_{k}}\sum_{j=1}^{k}\rho_{j}\gamma_{j}\left(\epsilon^x_{j} + \inpr{x_{j}}{b_{j}-\overline{b}_{k}}\right).
\end{align*}
Thus, adding $\overline{\epsilon}^x_{k}$ and $\overline{\epsilon}^y_{k}$ and combining with \eqref{Phi(y,b)} and the equation above, we deduce the third equality in \eqref{eq:rew-erg-ite}.
\end{proof}

The following result establishes bounds for the quantities $\overline{a}_k+\overline{b}_k$, $\overline{x}_k-\overline{y}_k$ and $\overline{\epsilon}^x_k+\overline{\epsilon}_k^y$.

\begin{theorem}
\label{Teo erg bound}
Assume the hypotheses of Lemma $\ref{Prop rewriten residuals}$ and let $d_{0}$ be the distance of $(z_{0},w_{0})$ to $\exset{A}{B}$. Then, for all integer $k\geq1$, we have
\begin{align}
\label{compl_erg_ax}
&\norm{\overline{a}_{k}+\overline{b}_{k}} \leq \frac{2d_{0}}{\Gamma_{k}},\qquad\qquad\norm{\overline{x}_{k}-\overline{y}_{k}} \leq  \frac{2d_{0}}{\Gamma_{k}},\\
\label{compl_erg_ep}
&\overline{\epsilon}^x_{k} + \overline{\epsilon}^y_{k} \leq \frac{1}{\Gamma_{k}}\left[\frac{1}{\Gamma_{k}}\sum_{j=1}^{k}\rho_{j}\gamma_{j}\norm{(y_{j},b_{j})-(z_{j-1},w_{j-1})}^{2}+4d_0^2\right].
\end{align}
\end{theorem}

\begin{proof}
We combine the first two identities in \eqref{eq:rew-erg-ite} with the second inequality in (\ref{dist_est}) to obtain
\begin{equation*}
\norm{(\overline{a}_{k}+\overline{b}_{k},\overline{x}_{k}-\overline{y}_{k})} = \frac{1}{\Gamma_{k}}\norm{(z_{0},w_{0})-(z_{k},w_{k})} \leq \frac{2d_{0}}{\Gamma_{k}}.
\end{equation*}
Thus, the bounds in \eqref{compl_erg_ax} follow immediately from the equation above.

Now, taking $(z,w)=(\overline{y}_k,\overline{b}_k)$ in equation \eqref{sum identity Phi} and rearranging the terms we have
\begin{align*}
-\sum_{j=1}^k\rho_j\gamma_j\phi_j(\overline{y}_k,\overline{b}_k) = & \,\,\dfrac{1}{2}\norm{(\overline{y}_k,\overline{b}_k)-(z_0,w_0)}^2 - \dfrac{1}{2}\norm{(\overline{y}_k,\overline{b}_k)-(z_k,w_k)}^2\\
&\,\, - \frac{1}{2}\sum_{j=1}^k\rho_j(2-\rho_j)\gamma_j^2\norm{\nabla\phi_j}^2.
\end{align*}
Since $\rho_j\in\left]0,2\right[$ for all integer $j\geq1$, the equation above implies
\begin{equation}
\label{eq:eq1-teor-erg-bound}
-\sum_{j=1}^k\rho_j\gamma_j\phi_j(\overline{y}_k,\overline{b}_k) \leq \,\,\dfrac{1}{2}\norm{(\overline{y}_k,\overline{b}_k)-(z_0,w_0)}^2.
\end{equation}
Next, we define $(\overline{z}_k,\overline{w}_k):=\dfrac{1}{\Gamma_k}\sum\limits_{j=1}^k\rho_j\gamma_j(z_{j-1},w_{j-1})$ and use the triangle inequality for norms to obtain
\begin{equation}
\label{eq:eq2-teor-erg-bound}
\begin{split}
\dfrac{1}{2}\norm{(\overline{y}_k,\overline{b}_k)-(z_0,w_0)}^2 \leq &\, \norm{(\overline{y}_k,\overline{b}_k)-(\overline{z}_k,\overline{w}_k)}^2 + \norm{(\overline{z}_k,\overline{w}_k)-(z_0,w_0)}^2\\
 \leq & \, \dfrac{1}{\Gamma_k}\sum_{j=1}^k\rho_j\gamma_j\norm{(y_j,b_j)-(z_{j-1},w_{j-1})}^2\\
&\,\,\, + \dfrac{1}{\Gamma_k}\sum_{j=1}^k\rho_j\gamma_j\norm{(z_{j-1},w_{j-1})-(z_0,w_0)}^2\\
\leq & \, \dfrac{1}{\Gamma_k}\sum_{j=1}^k\rho_j\gamma_j\norm{(y_j,b_j)-(z_{j-1},w_{j-1})}^2 + 4d_0^2,
\end{split}
\end{equation}
where the second and the third inequalities above are due to the convexity of $\norm{\cdot}^2$ and the second bound in \eqref{dist_est}, respectively.

Combining \eqref{eq:eq1-teor-erg-bound} with \eqref{eq:eq2-teor-erg-bound} we deduce that 
\begin{equation*}
-\sum_{j=1}^k\rho_j\gamma_j\phi_j(\overline{y}_k,\overline{b}_k) \leq \dfrac{1}{\Gamma_k}\sum_{j=1}^k\rho_j\gamma_j\norm{(y_j,b_j)-(z_{j-1},w_{j-1})}^2 + 4d_0^2.
\end{equation*}
Relation above, together with the last equality in \eqref{eq:rew-erg-ite}, now yields \eqref{compl_erg_ep}.
\end{proof}

\subsection{Stopping Criterion}
\label{subsec:stopping criterion}
In order to analyze the complexity properties of instances of Algorithm \ref{alg_genr_proj}, we define a termination condition for this general method in terms of the $\epsilon$-enlargements of operators $A$ and $B$. This criterion will enable the obtention of complexity bounds, proportional to the distance 
of the initial iterate to the extended solution set $S_e(A,B)$, for all the schemes presented in this work.

We consider the following stopping criterion for Algorithm \ref{alg_genr_proj}. Given an arbitrary pair of scalars $\delta$, $\epsilon>0$, Algorithm \ref{alg_genr_proj} will stop when it finds a
pair of points $(x,b,\epsilon^{x})$, $(y,a,\epsilon^{y})$ $\in\mathbb{E}$ such that 
\begin{align}
\label{stop_criterion}
b \in B^{\epsilon^{x}}(x), & & a \in A^{\epsilon^{y}}(y), & & \max\{\norm{a+b},\norm{x-y}\} \leq \delta, & & \max\{\epsilon^{x},\epsilon^{y}\} \leq \epsilon. 
\end{align}
We observe that if $\delta=\epsilon=0$, in view of Proposition \ref{Prop enlargement}, the above termination criterion is reduced to $b\in B(x)$, $a\in A(y)$, $x=y$ and $b=-a$, in which case $(x,b)\in S_e(A,B)$.

Based on the termination condition \eqref{stop_criterion} we can define the following notion of approximate solutions of problem \eqref{problem}.
\begin{definition}
\emph{
For a given pair of positive scalars $(\delta,\epsilon)$, a pair $(x,y)\in\R^n\times\R^n$ is called a $(\delta,\epsilon)$\emph{-approximate solution} (or $(\delta,\epsilon)$\emph{-solution}) of problem (\ref{problem}), if there exist $b,a\in\mathbb{R}^{n}$ and $\epsilon^{x},\epsilon^{y}\in\mathbb{R}_{+}$ such that the relations in (\ref{stop_criterion}) hold. }
\end{definition}

\section{The Family of Projective Splitting Methods}
\label{complexity}
Our goal in this section is to establish the complexity analysis of the family of PSM developed in 
\cite{eck_sv_08} for solving \eqref{problem}. First, we will observe that the PSM is an instance of the general 
Algorithm \ref{alg_genr_proj} with the feature of solving two proximal subproblems exactly, one involving only $A$ and the other one only $B$, for obtaining the decomposable separator in step 2 of Algorithm \ref{alg_genr_proj}. This will allow us to use the results of section \ref{sec:general framework} to derive
general iteration-complexity bounds for the PSM. Such bounds will be expressed in terms of the parameter sequences $\{\lambda_k\}$, 
$\{\mu_k\}$, $\{\rho_k\}$ and $\{\alpha_k\}$, calculated at each iteration of the method (see the PSM below). In 
subsection \ref{specialized}, we will specialize these results for the case where global convergence was obtained 
in \cite{eck_sv_08}.

We start by stating the family of projective splitting methods (PSM).
\begin{algorithmn}[\textbf{PSM}]
Choose $(z_{0},w_{0})\in\R^n\times\R^n$. For $k=1,2,\dots$
\begin{itemize}
\item[1.] Choose $\lambda_{k}$, $\mu_{k}>0$ and $\alpha_{k}\in\mathbb{R}$ such that 
\begin{equation}
\label{ass_con}
\frac{\mu_{k}}{\lambda_{k}}-\left(\frac{\alpha_{k}}{2}\right)^{2}>0,
\end{equation}
and find $(x_{k},b_{k})\in\graf{B}$ and $(y_{k},a_{k})\in\graf{A}$ such that
\begin{align}
\label{sub_prob_b}
\lambda_{k}b_{k} + x_{k} &= z_{k-1} + \lambda_{k}w_{k-1}, \\
\label{sub_prob_a}
\mu_{k}a_{k} + y_{k} &= (1-\alpha_{k})z_{k-1} + \alpha_{k}x_{k} - \mu_{k}w_{k-1}.
\end{align}
\item[2.] If $\norm{a_{k}+b_{k}} + \norm{x_{k}-y_{k}}=0$ stop. Otherwise, set
\begin{equation}
\label{def_gamma}
\gamma_{k} = \frac{\inpr{z_{k-1}-x_{k}}{b_{k}-w_{k-1}} + \inpr{z_{k-1}-y_{k}}{a_{k}+w_{k-1}}}{\norm{a_{k}+b_{k}}^{2} + \norm{x_{k}-y_{k}}^{2}}.
\end{equation}
\item[3.] Choose a parameter $\rho_{k}\in\left]0,2\right[$ and set
\begin{equation}
\label{upd_rule}
\begin{split}
z_{k} = & \,\,z_{k-1} - \rho_{k}\gamma_{k}(a_{k}+b_{k}),\\
w_{k} = & \,\,w_{k-1} - \rho_{k}\gamma_{k}(x_{k}-y_{k}).
\end{split}
\end{equation}
\end{itemize}
\end{algorithmn}

Several remarks are in order. The PSM is the same as \cite[Algorithm 2]{eck_sv_08}, except for the stopping criterion in step 2 above and boundedness 
conditions imposed on the parameters $\rho_{k}$, $\lambda_{k}$ and $\mu_{k}$ in \cite{eck_sv_08}.  Note that if $\norm{a_{k}+b_{k}}+\norm{x_{k}-y_{k}}=0$ for some $k$, then $x_{k}=y_{k}$, $b_{k}=-a_{k}$ and, since the points $(x_{k},b_{k})$ and $(y_{k},a_{k})$ are chosen in the graph of $B$ and $A$, respectively, we have $(x_{k},b_{k})\in\exset{A}{B}$. Therefore, when the PSM stops in step 2, it has found a point in the extended solution set.

Observe also that, since $B$ is maximal monotone, Minty's theorem \cite{minty} implies that the resolvent mappings $(I+\lambda_{k}B)^{-1}$ are everywhere defined and single valued for all integer $k\geq1$. Hence, by (\ref{sub_prob_b}) we have that the points $x_{k}=(I+\lambda_k B)^{-1}(z_{k-1}+\lambda_{k}w_{k-1})$ and $b_{k}=(1/\lambda_k)(z_{k-1}-x_k)+w_{k-1}$ exist and are unique. Similarly, the maximal monotonicity of $A$, together with (\ref{sub_prob_a}), guarantees the\,\,\, existence\,\,\, and\,\,\, uniqueness\,\,\, of\,\,\, $y_{k}=(I+\mu_kA)^{-1}((1-\alpha_k)z_{k-1}+\alpha_kx_k-\mu_kw_{k-1})$\,\,\, and $a_{k}=(1/\mu_k)((1-\alpha_k)z_{k-1}+\alpha_kx_k)-w_{k-1}$.

Moreover, if for $k=1,2,\ldots$, we denote by $\phi_{k}$ the decomposable separator (see Definition \ref{def_dec_sep_gen}) associated with the triplets $(x_{k},b_{k},0)$ and $(y_{k},a_{k},0)$, calculated in step 1 of the PSM, 
then the update rule in step 3 of the PSM can be restated as
\begin{equation*}
(z_{k},w_{k}) = (z_{k-1},w_{k-1}) - \rho_{k}\gamma_{k}\nabla\phi_{k}.
\end{equation*}
Consequently, the family of PSM falls within the general framework of Algorithm \ref{alg_genr_proj}, and the results of section \ref{sec:general framework} apply.

Finally, note that the PSM generates, on each iteration, a pair $(x_k,y_k)$ and vectors 
$b_k$, $a_k\in\R^n$ such that the inclusions in \eqref{stop_criterion} hold with $(x,b,\epsilon^x)=(x_k,b_k,0)$ and $(y,a,\epsilon^y)=(y_k,a_k,0)$. Hence,
we can try to develop bounds for the quantities $\norm{a_{k}+b_{k}}$ and $\norm{x_{k}-y_{k}}$ to estimate when an iterate $(x_k,y_k)$ is bound to satisfy the termination criterion \eqref{stop_criterion}.

Before establishing the iteration-complexity results for the PSM, we need the following technical result.  
It presents two lower bounds for $\phi_k(z_{k-1},w_{k-1})$.

\begin{lemma}
\label{Lem_est_phi}
Let $\{(x_{k},b_{k})\}$, $\{(y_{k},a_{k})\}$, $\{(z_{k},w_{k})\}$, $\{\lambda_k\}$, $\{\mu_k\}$, $\{\alpha_k\}$ and $\{\rho_k\}$ be the sequences generated by the PSM, and $\{\phi_k\}$ be the 
sequence of decomposable separators associated with the PSM. Then, for all integer $k\geq1$, the following inequalities hold
\begin{align}
\label{est_ph_1}
&\phi_{k}(z_{k-1},w_{k-1}) \geq \frac{\theta_{k}}{\delta_{k}}\left(\norm{a_{k}+b_{k}}^{2} + \norm{x_{k}-y_{k}}^{2}\right),\\
\label{est_ph_2}
&\phi_{k}(z_{k-1},w_{k-1}) \geq \frac{\theta_{k}}{\mu_{k}}\left(\norm{z_{k-1}-y_{k}}^{2} + \norm{w_{k-1}-b_{k}}^{2}\right),
\end{align}
where $\delta_{k}:=\mu_{k}+(1-\alpha_{k})\lambda_{k}>0$ and $\theta_{k}>0$ is the smallest eigenvalue of the matrix 
\begin{equation*}
\left(\begin{array}{cc}
1 & -\frac{\lambda_{k}|\alpha_{k}|}{2}\\
-\frac{\lambda_{k}|\alpha_{k}|}{2} & \lambda_{k}\mu_{k}
\end{array}\right).
\end{equation*}
\end{lemma}

\begin{proof}
Inequality (\ref{est_ph_1}) was obtained in \cite[Proposition 3]{eck_sv_08} as part of the convergence proof of Algorithm 2 in \cite{eck_sv_08}, as were the assertions that $\theta_{k},\,\delta_{k}>0$ 
under assumption (\ref{ass_con}). Therefore, we only need to prove here relation (\ref{est_ph_2}).

If we subtract $y_{k}$ from both sides of (\ref{sub_prob_b}) and rearrange the terms we obtain
\begin{equation}
\label{eql_1}
x_{k} - y_{k} = z_{k-1}-y_{k} + \lambda_{k}(w_{k-1}-b_{k}).
\end{equation}
Now, adding $\mu_{k}b_{k}$ to both sides of (\ref{sub_prob_a}) and rearranging we have
\begin{align}
\nonumber
\mu_{k}(a_{k}+b_{k}) &= (1-\alpha_{k})z_{k-1} + \alpha_{k}x_{k} - y_{k} - \mu_{k}(w_{k-1}-b_{k})\\
\label{eql_2}
&= \alpha_{k}(x_{k}-y_{k}) + (1-\alpha_{k})(z_{k-1}-y_{k}) - \mu_{k}(w_{k-1}-b_{k}).
\end{align}
Next, we substitute (\ref{eql_1}) into (\ref{eql_2}) and divide by $\mu_{k}$ to obtain
\begin{align}
\nonumber
a_{k}+b_{k} &= \frac{\alpha_{k}}{\mu_{k}}(z_{k-1}-y_{k} + \lambda_{k}(w_{k-1}-b_{k})) + \frac{(1-\alpha_{k})}{\mu_{k}}(z_{k-1}-y_{k}) - (w_{k-1}-b_{k})\\
\label{eql_3}
& = \frac{1}{\mu_{k}}(z_{k-1}-y_{k}) + \left(\frac{\alpha_{k}\lambda_{k}}{\mu_{k}}-1\right)(w_{k-1}-b_{k}).
\end{align}
Since
\begin{align*}
\phi_{k}(z_{k-1},w_{k-1}) = & \inpr{a_{k}+b_{k}}{z_{k-1}-y_{k}} + \inpr{x_{k}-y_{k}}{w_{k-1}-b_{k}},
\end{align*}
equation above, together with (\ref{eql_1}) and (\ref{eql_3}), yields
\begin{equation*}
\begin{split}
\phi_{k}(z_{k-1},w_{k-1}) = \,&\frac{1}{\mu_{k}}\norm{z_{k-1}-y_{k}}^{2} + \frac{\alpha_{k}\lambda_{k}}{\mu_{k}}\inpr{z_{k-1}-y_{k}}{w_{k-1}-b_{k}} + \lambda_{k}\norm{w_{k-1}-b_{k}}^{2}\\
\geq\, & \frac{1}{\mu_{k}}\norm{z_{k-1}-y_{k}}^{2} - \frac{|\alpha_{k}|\lambda_{k}}{\mu_{k}}\norm{z_{k-1}-y_{k}}\norm{w_{k-1}-b_{k}} + \lambda_{k}\norm{w_{k-1}-b_{k}}^{2}\\
=\,& \frac{1}{\mu_{k}}
\left(\begin{array}{c}
\norm{z_{k-1}-y_{k}}\\
\norm{w_{k-1}-b_{k}}
\end{array}\right)^{T}
\left(\begin{array}{cc}
1 & -\frac{\lambda_{k}|\alpha_{k}|}{2}\\
-\frac{\lambda_{k}|\alpha_{k}|}{2} & \lambda_{k}\mu_{k}
\end{array}\right)
\left(\begin{array}{c}
\norm{z_{k-1}-y_{k}}\\
\norm{w_{k-1}-b_{k}}
\end{array}\right),
\end{split}
\end{equation*}
where the inequality in the above relation follows from the Cauchy-Schwartz inequality. Finally, (\ref{est_ph_2}) follows from the expression above and the definition of $\theta_{k}$.
\end{proof}

For simplicity, the convergence rate results presented below suppose that the PSM never stops in step 2,
i.e. they are assuming that $\norm{\nabla\phi_{k}}>0$ for all integer $k\geq1$. However, they can easily be restated without assuming such condition by saying that either the conclusion stated below holds or  $(x_{k},b_{k})$ is a point in $\exset{A}{B}$.

Next result estimates the quality of the best iterate among $(x_{1},y_{1}),\ldots,(x_{k},y_{k})$ in terms of the stopping criterion \eqref{stop_criterion}. 
We refer to these estimates as \emph{pointwise} complexity bounds for the PSM. 

\begin{theorem}
\label{Lem_compl_it}
Let $\{(x_{k},b_{k})\}$, $\{(y_{k},a_{k})\}$, $\{\lambda_k\}$, $\{\mu_k\}$, $\{\alpha_k\}$, $\{\gamma_k\}$ and $\{\rho_k\}$ be the sequences generated by the PSM. 
Then, for every integer $k\geq1$, we have
\begin{equation}
\label{inclusions it}
b_{k}\in B(x_{k}),\qquad\qquad a_{k}\in A(y_{k}),
\end{equation}
and there exists an index $1\leq i \leq k$ such that
\begin{equation}
\label{compl_est_1}
\norm{a_{i}+b_{i}}^{2}+\norm{x_{i}-y_{i}}^{2} \leq \frac{d_{0}^{2}}{\sum\limits_{j=1}^{k}\rho_{j}(2-\rho_{j})\left(\frac{\theta_{j}}{\delta_{j}}\right)^2},
\end{equation}
where $d_{0}$ is the distance of the first iterate $(z_{0},w_{0})$ to $\exset{A}{B}$, and $\theta_k$ and $\delta_k$ are defined in Lemma $\ref{Lem_est_phi}$.
\end{theorem}

\begin{proof}
The assertions that $b_{k}\in B(x_{k})$ and $a_{k}\in A(y_{k})$ are direct consequences of step 1 in the PSM.
The definition of $\gamma_{j}$ in step 2 of the method, together with the definition of $\phi_{j}$  and inequality (\ref{est_ph_1}), yields
\begin{equation}
\label{est_gamma_j}
\gamma_{j}=\frac{\phi_{j}(z_{j-1},w_{j-1})}{\norm{\nabla\phi_{j}}^{2}} \geq \frac{\theta_{j}}{\delta_{j}}\qquad\qquad \text{for }j=1,2,\ldots.
\end{equation}
Therefore,
\begin{equation*}
\gamma_j^2\geq\left(\dfrac{\theta_j}{\delta_j}\right)^2\quad\qquad\qquad\text{for }j=1,2,\dots.
\end{equation*}
Multiplying both sides of the inequality above by $\rho_j(2-\rho_j)\norm{\nabla\phi_j}^2$, adding from $j=1$ to $k$ and using the first bound in (\ref{dist_est}), we have
\begin{equation*}
d_{0}^{2} \geq \sum_{j=1}^{k}\norm{\nabla\phi_{j}}^{2}\rho_{j}(2-\rho_{j})\left(\frac{\theta_{j}}{\delta_{j}}\right)^2.
\end{equation*}
Taking $i$ such that
\begin{equation*}
 i\in\arg\min_{j=1,\ldots,k}\left(\norm{\nabla\phi_{j}}^{2}\right),
\end{equation*}
and using the previous inequality we obtain
\begin{equation*}
 d_{0}^{2} \geq \sum_{j=1}^{k}\rho_{j}(2-\rho_{j})\left(\frac{\theta_{j}}{\delta_{j}}\right)^2\norm{\nabla\phi_i}^2.
\end{equation*}
Bound \eqref{compl_est_1} now follows from the above relation and noting that $\nabla\phi_{i}=(a_{i}+b_{i},x_{i}-y_{i})$.
\end{proof}

We will now derive alternative complexity bounds for the PSM. 
Using the sequences of ergodic iterates associated with the PSM, defined as in subsection \ref{ergodic iterates}, we will obtain convergence rates for the methods in the ergodic sense.  We refer to these kind of estimates as \emph{ergodic} complexity bounds.

Define the sequences of ergodic means $\{(\overline{x}_{k},\overline{b}_k,\overline{\epsilon}^x_{k})\}$ and $\{(\overline{y}_{k},\overline{a}_k,\overline{\epsilon}^y_{k})\}$, associated with
the sequences $\{(x_{k},b_k,0)\}$, $\{(y_{k},a_k,0)\}$, $\{\gamma_k\}$ and $\{\rho_k\}$ generated by the PSM, as in \eqref{def_xy_erg}, \eqref{def_res_erg_x} and \eqref{def_res_erg_y}.
According to Lemma \ref{Lem_residual_erg}, we can attempt to bound the size of $\norm{\overline{a}_{k}+\overline{b}_{k}}$, $\norm{\overline{x}_{k}-\overline{y}_{k}}$, $\overline{\epsilon}^x_{k}$ and 
$\overline{\epsilon}^y_{k}$, in order to know when the ergodic iterates $\{\overline{x}_{k}\}$ and $\{\overline{y}_{k}\}$ will meet the stopping condition \eqref{stop_criterion}.

\begin{theorem}
\label{Teo_compl_erg}
Assume the hypotheses of Theorem $\ref{Lem_compl_it}$. In addition, consider the sequences of ergodic iterates $\{(\overline{x}_{k},\overline{b}_k,\overline{\epsilon}^x_{k})\}$ and
$\{(\overline{y}_{k},\overline{a}_k,\overline{\epsilon}^y_{k})\}$ associated with the sequences generated by the PSM, defined as in \eqref{def_xy_erg}, \eqref{def_res_erg_x} and \eqref{def_res_erg_y}. Then, for every integer $k\geq1$, we have
\begin{equation}
\label{inclusions erg}
\overline{b}_k\in B^{\overline{\epsilon}^x_{k}}(\overline{x}_k),\qquad\qquad \overline{a}_k\in A^{\overline{\epsilon}^y_{k}}(\overline{y}_k),
\end{equation}
and
\begin{align}
\label{compl_erg}
&\norm{\overline{a}_{k}+\overline{b}_{k}} \leq \frac{2d_{0}}{\Gamma_{k}}, &\norm{\overline{x}_{k}-\overline{y}_{k}} \leq  \frac{2d_{0}}{\Gamma_{k}},& &\overline{\epsilon}^x_{k} + \overline{\epsilon}^y_{k} \leq \frac{d_{0}^{2}(\varsigma_{k}+4)}{\Gamma_{k}},
\end{align}
where 
\begin{equation*}
\varsigma_{k}:=\max_{j=1,\ldots,k}\left\lbrace\frac{\mu_{j}}{\theta_{j}(2-\rho_{j})\Gamma_{k}}\right\rbrace.
\end{equation*}
\end{theorem}

\begin{proof}
Inclusions in \eqref{inclusions erg} are a consequence of Lemma \ref{Lem_residual_erg}. The first two inequalities in \eqref{compl_erg} are obtained by applying Theorem \ref{Teo erg bound}.

Now, we observe that relation (\ref{est_ph_2}), together with the equality in \eqref{est_gamma_j}, implies
\begin{equation*}
\frac{\mu_{j}}{\theta_{j}}\norm{\nabla\phi_{j}}^{2}\gamma_{j} \geq \norm{(y_{j},b_{j})-(z_{j-1},w_{j-1})}^{2}\qquad\quad \text{for } j=1,2,\dots.
\end{equation*}
Multiplying the inequality above by $\dfrac{1}{\Gamma_k}\rho_j\gamma_j$ and adding from $j=1$ to $k$, we obtain
\begin{align*}
\frac{1}{\Gamma_{k}}\sum_{j=1}^{k}\rho_{j}\gamma_{j}\norm{(y_{j},b_{j})-(z_{j-1},w_{j-1})}^{2} \leq &\, \frac{1}{\Gamma_{k}}\sum_{j=1}^{k}\frac{\mu_{j}}{\theta_{j}}\rho_{j}\gamma_{j}^{2}\norm{\nabla\phi_{j}}^{2}\\
= &\, \frac{1}{\Gamma_{k}}\sum_{j=1}^{k}\frac{\mu_{j}}{\theta_{j}(2-\rho_j)}\rho_{j}(2-\rho_j)\gamma_{j}^{2}\norm{\nabla\phi_{j}}^{2}\\
\leq &\, \left(\max_{j=1,\ldots,k}\left\lbrace\frac{\mu_{j}}{\theta_{j}(2-\rho_{j})\Gamma_{k}}\right\rbrace\right)\sum_{j=1}^{k}\rho_{j}(2-\rho_{j})\gamma_{j}^{2}\norm{\nabla\phi_{j}}^{2}\\
\leq & \,\left(\max_{j=1,\ldots,k}\left\lbrace\frac{\mu_{j}}{\theta_{j}(2-\rho_{j})\Gamma_{k}}\right\rbrace\right)d_{0}^{2},
\end{align*}
where the last inequality above follows from the first estimate in (\ref{dist_est}). We combine the relation above with (\ref{compl_erg_ep}) and the definition 
of $\varsigma_{k}$ to deduce the last bound in \eqref{compl_erg}.
\end{proof}

\subsection{Specialized Complexity Results}
\label{specialized}  
In this subsection, we will specialize the general pointwise and ergodic complexity bounds derived for the PSM in Theorems \ref{Lem_compl_it} and \ref{Teo_compl_erg}, respectively, for the case where global convergence was obtained in \cite{eck_sv_08}.

In \cite{eck_sv_08}, it was proven convergence of the sequences $\{(x_{k},b_{k})\}$, $\{(y_{k},-a_{k})\}$ and $\{(z_{k},w_{k})\}$, calculated by the PSM, to a point
of $S_{e}(A,B)$ under the following assumptions:

\begin{enumerate}
\item[(A.1)] there exist $\underline{\lambda}$ and $\overline{\lambda}$ such that, $\overline{\lambda}\geq\underline{\lambda}>0$ and $\lambda_{k},\mu_{k}\in[\underline{\lambda},\overline{\lambda}]$ for all integer $k\geq1$;
\item[(A.2)] there exists $\overline{\rho}\in\left[0,1\right[$ such that $\rho_{k}\in[1-\overline{\rho},1+\overline{\rho}]$ for all integer $k\geq1$;
\item[(A.3)] $\nu:=\inf\limits_{k}\left\lbrace\dfrac{\mu_{k}}{\lambda_{k}}-\left(\dfrac{\alpha_{k}}{2}\right)^{2}\right\rbrace>0$.
\end{enumerate}

Under hypotheses (A.1)-(A.3) we will show that the PSM has $\mathcal{O}(1/\sqrt{k})$ pointwise convergence
rate, while the rate in the ergodic sense is $\mathcal{O}(1/k)$.

\begin{theorem}
\label{Teo compl it esp}
Let $\{(x_{k},b_{k})\}$, $\{(y_{k},a_{k})\}$, $\{\lambda_k\}$, $\{\mu_k\}$, $\{\alpha_k\}$, $\{\gamma_k\}$ and $\{\rho_k\}$ be the sequences generated by the PSM under assumptions \emph{(A.1)-(A.3)}.
If $d_{0}$ denote the distance of $(z_{0},w_{0})$ to the extended solution set $\exset{A}{B}$. Then, for all integer $k\geq1$, we have
\begin{equation*}
b_k\in B(x_k),\qquad\qquad a_k\in A(y_k),
\end{equation*}
and there exists an index $1\leq i\leq k$ such that
\begin{align*}
\norm{a_{i}+b_{i}} \leq \frac{d_{0}\upsilon}{\sqrt{k}(1-\overline{\rho})}\qquad\quad\text{and}\qquad\quad
\norm{x_{i}-y_{i}} \leq \frac{d_{0}\upsilon}{\sqrt{k}(1-\overline{\rho})},
\end{align*}
where 
\begin{equation*}
\upsilon:=\frac{2\overline{\lambda}\left(1+\overline{\lambda}^2\right)\left(1+\sqrt{\overline{\lambda}/\underline{\lambda}}\right)}{\underline{\lambda}^{2}\nu}.
\end{equation*}
\end{theorem}

\begin{proof}
The inclusions in the statement of the theorem follow from \eqref{inclusions it}.
Now, we note that condition (A.2) implies
\begin{equation}
\label{rho_est}
\rho_{j}(2-\rho_{j})\geq(1-\overline{\rho})^{2}\qquad\qquad\text{for }\,\,j=1,2,\dots.
\end{equation}
Next, we observe that relation (\ref{ass_con}) in step 1 of the PSM yields $|\alpha_{j}|\leq2\sqrt{\mu_{j}/\lambda_{j}}$. Hence, assumption (A.1) implies 
\begin{equation*}
|\alpha_{j}|\leq2\sqrt{\overline{\lambda}/\underline{\lambda}}\quad\qquad\qquad\text{for }\,\,j=1,2,\dots.
\end{equation*}
The inequality above, together with the definition of $\delta_{j}$ in Lemma \ref{Lem_est_phi} and assumption (A.1), yields
\begin{equation}
\label{delta_est}
\delta_{j} \leq 2\overline{\lambda}\left(1+\sqrt{\overline{\lambda}/\underline{\lambda}}\right).
\end{equation}
Moreover, in \cite[Proposition 3]{eck_sv_08} it was proven that
\begin{align}
\nonumber
\theta_{j}: = &\,  \frac{1}{2}\left(1+\lambda_{j}\mu_{j}-\sqrt{(1+\lambda_{j}\mu_{j})^{2}-4(\lambda_{j}\mu_{j}-\left(\lambda_{j}\alpha_{j}/2)^{2}\right)}\right)\\
\nonumber
 \geq &\,\frac{\lambda_{j}^{2}(\mu_{j}/\lambda_{j}-(\alpha_{j}/2)^{2})}{1+\lambda_{j}\mu_{j}}.
\end{align}
Thus, under hypotheses (A.1)-(A.3) we have 
\begin{equation}
\label{theta_est}
\theta_j \geq \frac{\underline{\lambda}^{2}\nu}{1+\overline{\lambda}^2},
\end{equation}
and combining \eqref{delta_est} with \eqref{theta_est} we obtain
\begin{equation}\label{eq:thetha/delta-est}
\frac{\theta_j}{\delta_j}\geq\frac{\underline{\lambda}^2\nu}{\left(1+\overline{\lambda}^2\right)2\overline{\lambda}\left(1+\sqrt{\overline{\lambda}/\underline{\lambda}}\right)}=\frac{1}{\upsilon}\qquad\qquad\text{for }\,\,j=1,2,\ldots.
\end{equation}
Now, from inequalities \eqref{eq:thetha/delta-est} and (\ref{rho_est}) we deduce that
\begin{equation*}
\rho_{j}(2-\rho_{j})\left(\frac{\theta_{j}}{\delta_{j}}\right)^2\geq \frac{(1-\overline{\rho})^{2}}{\upsilon^2}
\end{equation*}
for $j=1,\ldots,k$. Hence, adding equation above from $j=1$  to $k$ we have
\begin{equation*}
\sum_{j=1}^{k}\rho_{j}(2-\rho_{j})\left(\frac{\theta_{j}}{\delta_{j}}\right)^2\geq k\frac{(1-\overline{\rho})^{2}}{\upsilon^2}.
\end{equation*}
The theorem follows combining the above expression with inequality (\ref{compl_est_1}).
\end{proof}

\begin{theorem}
\label{teo:erg-comp-esp}
Assume the hypotheses of Theorem $\ref{Teo compl it esp}$. Consider the sequences of ergodic iterates $\{(\overline{x}_{k},\overline{b}_{k},\overline{\epsilon}^x_{k})\}$ and $\{(\overline{y}_{k},\overline{a}_{k},\overline{\epsilon}^y_{k})\}$ associated with the sequences generated by the PSM, defined as in \eqref{def_xy_erg}, \eqref{def_res_erg_x} and \eqref{def_res_erg_y}. Then, for every integer $k\geq1$, we have 
\begin{equation}
\label{inclusion_erg}
\overline{b}_{k}\in B^{\overline{\epsilon}^x_{k}}(\overline{x}_{k}),\qquad\qquad \overline{a}_{k}\in A^{\overline{\epsilon}^y_{k}}(\overline{y}_{k}),
\end{equation}
and
\begin{align}
\label{ab}
\norm{\overline{a}_{k}+\overline{b}_{k}} \leq \frac{2d_{0}\upsilon}{k(1-\overline{\rho})},\qquad\norm{\overline{x}_{k}-\overline{y}_{k}} \leq \frac{2d_{0}\upsilon}{k(1-\overline{\rho})},\qquad
\overline{\epsilon}^x_{k} + \overline{\epsilon}^y_{k} \leq \frac{d_{0}^{2}\upsilon(\upsilon'_{k}+4)}{k(1-\overline{\rho})},
\end{align}
where $$\upsilon'_{k}:=\frac{\overline{\lambda}\left(1+\overline{\lambda}^2\right)\upsilon}{{\underline{\lambda}^2}\nu(1-\overline{\rho})^{2}k}.$$
\end{theorem}

\begin{proof}
The inclusions in (\ref{inclusion_erg}) follow from Lemma \ref{Lem_residual_erg}. The definition of $\Gamma_{k}$, together with assumption (A.2) and equation (\ref{est_gamma_j}), yields
\begin{equation*}
\Gamma_{k}\geq (1-\overline{\rho})\sum_{j=1}^{k}\frac{\theta_{j}}{\delta_{j}}.
\end{equation*}
Therefore, by \eqref{eq:thetha/delta-est} and the inequality above we have
\begin{equation}
\label{gamma_bound}
\Gamma_{k} \geq  \frac{(1-\overline{\rho})k}{\upsilon}.
\end{equation}
The first two bounds in (\ref{ab}) now follow from (\ref{gamma_bound}) and the first two inequalities in (\ref{compl_erg}).

To conclude the proof we observe that the definition of $\varsigma_{k}$, hypothesis (A.1), (\ref{theta_est}) and (\ref{gamma_bound}) imply 
\begin{equation*}
\varsigma_k \leq \frac{\overline{\lambda}(1+\overline{\lambda}^2)\upsilon}{{\underline{\lambda}^2}\nu(1-\overline{\rho})^{2}k}.
\end{equation*}
Thus, combining the above relation with the last inequality in (\ref{compl_erg}), the definition of $\upsilon'_{k}$ and (\ref{gamma_bound}), we obtain the last bound in (\ref{ab}).
\end{proof}

We emphasize here that the derived bounds obtained in Theorem \ref{teo:erg-comp-esp} are asymptotically better than the ones obtained in Theorem \ref{Teo compl it esp}. Indeed, the bounds for $x_k$, $y_k$, $a_k$ and $b_k$ are $\mathcal{O}(1/\sqrt{k})$, whereas for $\overline{x}_k$, $\overline{y}_k$, $\overline{b}_k$, $\overline{a}_k$, $\overline{\epsilon}_k^x$ and $\overline{\epsilon}_k^y$ the bounds are $\mathcal{O}(1/k)$. However, the iterates calculated by the PSM are points in the graph of $A$ and $B$, while the ergodic iterates are in some outer approximation of the operators, namely they are points in an $\epsilon$-enlargement of $A$ and $B$.

The following result, which is an immediate consequence of Theorems \ref{Teo compl it esp} and \ref{teo:erg-comp-esp}, presents complexity bounds for the PSM to obtain $(\delta,\epsilon)$-approximate solutions of problem \eqref{problem}.

\begin{corollary}
Assume the hypotheses of Theorem $\ref{teo:erg-comp-esp}$. Then, the following statements hold.
\begin{itemize}
 \item[(a)] For every $\delta>0$ there exists an index 
 \begin{equation*}
 i = \mathcal{O}\left(\frac{d_{0}^{2}\upsilon^2}{\delta^{2}}\right)
 \end{equation*}
such that the iterate $(x_i,y_i)$ is a $(\delta,0)$-solution of problem \eqref{problem}. 
 \item[(b)] For every $\delta$, $\epsilon>0$ there exists an index
 \begin{equation*}
 k_{0} = \mathcal{O}\left(\max\left\lbrace \frac{d_{0}\upsilon}{\delta},\frac{d_{0}^{2}\upsilon}{\epsilon} \right\rbrace\right)
 \end{equation*}
 such that, for any $k \geq k_{0}$, the ergodic iterate $(\overline{x}_{k},\overline{y}_k)$ is a $(\delta,\epsilon)$-solution of problem \eqref{problem}. 
 \end{itemize}
\end{corollary}

\section{Spingarn's Splitting Method}
\label{sec:spingarn}
In  this section, we study the iteration-complexity 
of the two-operator case of 
Spingarn's splitting algorithm. 
In \cite{spingarn}, Spingarn describes a partial inverse method for solving MIPs given by the sum of $m$ maximal monotone operators.
Spingarn's method computes, at each iteration, 
independent proximal subproblems on each of the $m$ operators involved in the problem and then finds the next iterate by essentially averaging the results. 
This algorithm is actually a special case of the Douglas-Rachford splitting method \cite{eck_ber_dr}, 
and it is also a particular instance of the general projective splitting methods for sums of $m$ maximal monotone operators, which were introduced in \cite{eck_sv_09}. 

Eckstein and Svaiter proved in \cite{eck_sv_08} that the $m=2$ case of the Spingarn splitting method is a special case of a scaled variant of the PSM. 
Interpreting Spingarn's algorithm as an instance of the PSM allows us to use the analysis developed in the previous section for obtaining its complexity bounds. 

For this purpose, let us begin with a brief discussion of the reformulation of problem (\ref{problem}) studied in \cite{eck_sv_08}, obtained via including a scale factor.
If $\eta>0$ is a fixed scalar, multiplying both sides of (\ref{problem}) by $\eta$ gives the problem 
\begin{equation*}
0\in\eta A(z) + \eta B(z).
\end{equation*}
This simple reformulation leaves the solution set unchanged, but it transforms the set $S_e(A,B)$. Indeed,
the extended solution set associated with operators $\eta A$ and $\eta B$ has the form
\begin{equation*}
\exset{\eta A}{\eta B} = \{(z,\eta w):(z,w)\in\exset{A}{B}\}.
\end{equation*}

If we apply the PSM to $\eta A$ and $\eta B$, and consider $\eta a_{k}$, $\eta b_{k}$ and $\eta w_{k}$, respectively, in place of $a_{k}$, $b_{k}$ and $w_{k}$, after some algebraic manipulations we obtain a scheme identical to the PSM, except that (\ref{sub_prob_b})-(\ref{upd_rule}) are modified to
\begin{align}
\label{ref_b}
& \lambda_{k}\eta b_{k} + x_{k} = z_{k-1} + \lambda_{k}\eta w_{k-1}, && b_{k}\in B(x_{k}),\\
\label{ref_a}
& \mu_{k}\eta a_{k} + y_{k} = (1-\alpha_{k})z_{k-1} + \alpha_{k}x_{k} - \mu_{k}\eta w_{k-1}, && a_{k}\in A(y_{k}),\\
\label{ref_gam}
& \gamma_{k}=\frac{\inpr{z_{k-1}-x_{k}}{b_{k}-w_{k-1}} + \inpr{z_{k-1}-y_{k}}{a_{k}+w_{k-1}}}{\eta\norm{a_{k}+b_{k}}^{2} + \frac{1}{\eta}\norm{x_{k}-y_{k}}^{2}},\\
\label{ref_z}
& z_{k} =  z_{k-1} - \rho_{k}\gamma_{k}\eta(a_{k}+b_{k}),\\
\label{ref_w}
& w_{k} =  w_{k-1} - \frac{\rho_{k}\gamma_{k}}{\eta}(x_{k}-y_{k}).
\end{align}
The general pointwise and ergodic complexity bounds for the method above are obtained as a direct consequence of Theorems \ref{Lem_compl_it} and \ref{Teo_compl_erg}, replacing $a_{i}$, $b_{i}$, $\overline{a}_{k}$, $\overline{b}_{k}$, $\overline{\epsilon}^x_{k}$ and $\overline{\epsilon}^y_{k}$ by $\eta a_{i}$, $\eta b_{i}$, $\eta\overline{a}_{k}$, $\eta\overline{b}_{k}$, $\eta\overline{\epsilon}^x_{k}$ and $\eta\overline{\epsilon}^y_{k}$, respectively.

If $\eta>0$, in our notation, the Spingarn splitting method is reduced to the following set of recursions:
\begin{align}
\label{sping_b}
& \eta b_{k} + x_{k} = z_{k-1} + \eta w_{k-1}, && b_{k}\in B(x_{k}),\\
\label{sping_a}
& \eta a_{k} + y_{k} = z_{k-1} - \eta w_{k-1}, && a_{k}\in A(y_{k}),\\
\label{sping_z+}
& z_{k} =  (1-\rho_{k})z_{k-1} + \frac{\rho_{k}}{2}(x_{k}+y_{k}),\\
\label{sping_w+}
& w_{k} = (1-\rho_{k})w_{k-1} + \frac{\rho_{k}}{2}(b_{k}-a_{k}).
\end{align}

Note that if we take $\lambda_{k}=\mu_{k}=1$ and $\alpha_{k}=0$ in \eqref{ref_b}-\eqref{ref_w} for all integer $k\geq1$, then (\ref{ref_b})-(\ref{ref_a}) and (\ref{sping_b})-(\ref{sping_a}) are identical.
Moreover, the remaining calculations, (\ref{ref_gam}), (\ref{ref_z}) and (\ref{ref_w}), can be rewritten into the form (\ref{sping_z+})-(\ref{sping_w+}), as it is established in the next result.

\begin{theorem}
\label{Teo_sping}
Assume that the following condition is satisfied:
\begin{enumerate}
\item[\emph{(B)}]\qquad $\lambda_{k}=\mu_{k}=1$ and $\alpha_{k}=0$ in \eqref{ref_b}-\eqref{ref_w} for every integer $k\geq1$.
\end{enumerate}
Then, the recursions \eqref{ref_b}-\eqref{ref_w} and \eqref{sping_b}-\eqref{sping_w+} are identical. Hence, Spingarn's method is a special case of the PSM.
\end{theorem}
\begin{proof}
This result was proven in \cite[Subsection 4.2]{eck_sv_08}.
\end{proof}

The following theorem derives the global convergence rate of Spingarn's splitting method in terms of the termination criterion (\ref{stop_criterion}).

\begin{theorem}
Let $\eta>0$ and let $\{(x_{k},b_{k})\}$, $\{(y_{k},a_{k})\}$ and $\{\rho_k\}$ be the sequences generated by Spingarn's splitting method \eqref{sping_b}-\eqref{sping_w+}. 
For every $k\geq1$, define
\begin{equation}
\label{Gam sping}
P_k := \sum_{j=1}^k\rho_j,
\end{equation}
and
\begin{align}
\label{x b eps erg sping}
&\overline{x}_k := \frac{1}{P_k}\sum_{j=1}^k\rho_jx_j,\qquad \overline{b}_k := \frac{1}{P_k}\sum_{j=1}^k\rho_jb_j,\qquad\overline{\epsilon}^x_{k}:=\frac{1}{P_k}\sum_{j=1}^k\rho_j\inpr{x_j-\overline{x}_k}{b_j},\\
\label{y a eps erg sping}
&\overline{y}_k := \frac{1}{P_k}\sum_{j=1}^k\rho_jy_j,\qquad \overline{a}_k := \frac{1}{P_k}\sum_{j=1}^k\rho_ja_j,\qquad\overline{\epsilon}^y_{k}:=\frac{1}{P_k}\sum_{j=1}^k\rho_j\inpr{y_j-\overline{y}_k}{a_j}.
\end{align}
Assume hypothesis \emph{(A.2)} and set $d_{0}:=\dist{(z_{0},\eta w_{0})}{\exset{\eta A}{\eta B}}$. Then, the following statements hold.
\begin{itemize}
\item[(a)] For every integer $k\geq1$ we have
\begin{equation*}
b_k\in B(x_k),\qquad\qquad a_k\in A(y_k),
\end{equation*}
and there exists an index $1\leq i\leq k$ such that 
\begin{equation*}
\norm{a_{i}+b_{i}}\leq \frac{2d_{0}}{\eta\sqrt{k}(1-\overline{\rho})},\qquad\qquad \norm{x_{i}-y_{i}} \leq \frac{2d_{0}}{\sqrt{k}(1-\overline{\rho})}.
\end{equation*}
\item[(b)] For every integer $k\geq1$ we have 
$$\overline{b}_{k}\in B^{\overline{\epsilon}^x_{k}}(\overline{x}_{k}),\qquad\qquad \overline{a}_{k}\in A^{\overline{\epsilon}^y_{k}}(\overline{y}_{k}),$$ 
and
\begin{align*}
&\norm{\overline{a}_{k}+\overline{b}_{k}} \leq \frac{4d_{0}}{\eta k(1-\overline{\rho})},\qquad\qquad\qquad\norm{\overline{x}_{k}-\overline{y}_{k}} \leq  \frac{4d_{0}}{k(1-\overline{\rho})},\\
&\overline{\epsilon}^x_{k} + \overline{\epsilon}^y_{k} \leq \frac{2d_{0}^{2}}{\eta k(1-\overline{\rho})}\left(\frac{2}{(1-\overline{\rho})^2}+4\right).
\end{align*}
\end{itemize} 
\end{theorem}

\begin{proof}
\item[(a)] By the definitions of $\delta_{k}$ and $\theta_{k}$ in the statement of Lemma \ref{Lem_est_phi} and hypothesis (B), we have that $\delta_{k}=2$ and $\theta_{k}=1$ for every integer $k\geq1$. 

Therefore, since Theorem \ref{Teo_sping} gives that Spingarn's algorithm is a special case of (\ref{ref_b})-(\ref{ref_w}) under assumption (B), the claims in (a) follow from assumption (A.2) and Theorem \ref{Lem_compl_it} applied to (\ref{ref_b})-(\ref{ref_w}) with $\alpha_{k}=0$ and $\lambda_k=\mu_{k}=1$.

\item[(b)] The first assertions in (b) follow from the definitions of $P_k$, $\overline{x}_k$, $\overline{b}_k$, $\overline{\epsilon}^x_{k}$, $\overline{y}_k$, $\overline{a}_k$, $\overline{\epsilon}^y_{k}$ in \eqref{Gam sping}, \eqref{x b eps erg sping} and \eqref{y a eps erg sping},
the inclusions in (\ref{sping_b}), (\ref{sping_a}) and Theorem \ref{Teo Transp For}.

Now, we observe that Theorem \ref{Teo_sping} implies that the sequences $\{(x_k,\eta b_k)\}$ and $\{(y_k, \eta a_k)\}$ can be viewed as generated by the PSM applied to the operators 
$\eta A$ and $\eta B$, with $\lambda_k=\mu_k=1$ and $\alpha_k=0$ for $k=1,2,\dots$.
Moreover, in \cite[Subsection 4.2]{eck_sv_08} it was proven under assumption (B) that $\gamma_k$, given in \eqref{ref_gam}, is equal to $1/2$ for every integer $k\geq1$. Therefore,  the sequences of ergodic iterates associated
with $\{(x_k,\eta b_k)\}$, $\{(y_k, \eta a_k)\}$, $\{\rho_k\}$ and $\{\gamma_k\}$, which are obtained by equations \eqref{def_xy_erg}, \eqref{def_res_erg_x} and \eqref{def_res_erg_y} 
with $\Gamma_k=(1/2)P_k$, are exactly as defined in \eqref{x b eps erg sping} and \eqref{y a eps erg sping}, but with $\eta\overline{b}_k$,
$\eta\overline{\epsilon}^x_{k}$, $\eta\overline{a}_k$ and $\eta\overline{\epsilon}^y_{k}$ instead of $\overline{b}_k$,
$\overline{\epsilon}^x_{k}$, $\overline{a}_k$ and $\overline{\epsilon}^y_{k}$, respectively.

Hence, applying Theorem \ref{Teo_compl_erg} we have 
\begin{align}
\label{compl_erg_sping}
&\norm{\eta(\overline{a}_{k}+\overline{b}_{k})} \leq \frac{2d_{0}}{(1/2)P_{k}}, &\norm{\overline{x}_{k}-\overline{y}_{k}} \leq  \frac{2d_{0}}{(1/2)P_{k}},& &\eta(\overline{\epsilon}^x_{k} + \overline{\epsilon}^y_{k}) \leq \frac{d_{0}^{2}(\varsigma_{k}+4)}{(1/2)P_{k}},
\end{align}
where $d_0$ is the distance of $(z_0,\eta w_0)$ to $S_e(\eta A,\eta B)$ and $\varsigma_{k}=\max\limits_{j=1,\ldots,k}\left\lbrace\dfrac{\mu_{j}}{\theta_{j}(2-\rho_{j})(1/2)P_{k}}\right\rbrace$.

Next, we note that condition (A.2) yields $\rho_j \geq 1-\overline{\rho}$\, for every $j$,
therefore by the definition of $P_k$ we have 
\begin{equation}
\label{spin_gamma}
P_{k}\geq k(1-\overline{\rho}).
\end{equation} 
Furthermore, since in this case $\mu_{j}=1$, $\theta_{j}=1$\, and \,$2-\rho_j\geq 1-\overline{\rho}$\, for all integer\, $j\geq1$, the definition of $\varsigma_k$ and \eqref{spin_gamma} imply
\begin{equation*}
 \varsigma_k\leq \frac{2}{(1-\overline{\rho})^2k}\leq\frac{2}{(1-\overline{\rho})^2},\qquad\qquad\text{for }\,\, k=1,2,\dots.
\end{equation*}
Hence, the remaining claims in (b) follow combining the bounds in \eqref{compl_erg_sping} with the inequality above and \eqref{spin_gamma}.
\end{proof}

\begin{corollary}
Consider the sequences $\{x_k\}$ and $\{y_k\}$ generated by Spingarn's method and the sequences $\{\overline{x}_k\}$ and $\{\overline{y}_k\}$ defined in \eqref{x b eps erg sping} and \eqref{y a eps erg sping}, respectively. Then, the following statements hold.
\begin{itemize}
 \item[(a)] For every $\delta>0$ there exists an index 
 \begin{equation*}
 i = \mathcal{O}\left(\max\left\lbrace\frac{d_{0}^{2}}{\eta^2\delta^{2}},\frac{d_0^2}{\delta^2}\right\rbrace\right)
 \end{equation*}
such that the pair $(x_i,y_i)$ is a $(\delta,0)$-solution of problem \eqref{problem}. 
 \item[(b)] For every $\delta$, $\epsilon>0$ there exists an index
 \begin{equation*}
 k_{0} = \mathcal{O}\left(\max\left\lbrace\frac{d_0}{\eta\delta}, \frac{d_{0}}{\delta},\frac{d_{0}^{2}}{\eta\epsilon} \right\rbrace\right)
 \end{equation*}
 such that, for any $k \geq k_{0}$, the pair $(\overline{x}_{k},\overline{y}_k)$ is a $(\delta,\epsilon)$-solution of problem \eqref{problem}. 
 \end{itemize}
\end{corollary}

\section{A Parallel Inexact Case}
\label{inexact}

The PSM has to solve two proximal subproblems on each iteration, in order to construct decomposable separators.
Since finding the exact solution of subproblems (\ref{sub_prob_b}) and (\ref{sub_prob_a}) could be a challenging task, one might wish to allow approximate evaluations of
the resolvent mappings, while maintaining convergence of the method.

Our main goal in this section is to propose an inexact version of the PSM in the special case of taking $\alpha_k=0$ for all iteration $k$, which possibly allows the subproblems to be performed in parallel. 

It is customary to appeal to the theory of approximation criteria for the PPA and related methods, when attempting to approximate solutions of proximal subproblems.
The first inexact versions of the PPA were introduced in \cite{rock_ppa} by Rockafellar and are based on \emph{absolute summable} error criteria.
For instance, one of the approximation conditions proposed in \cite{rock_ppa} is
\begin{equation*}
\norm{z_{k+1}-(I+\lambda_{k}T)^{-1}(z_{k})}\leq s_{k}, \qquad\qquad\sum_{k=1}^{\infty}s_{k}<\infty.
\end{equation*}
This kind of approximation criteria, which involves a theoretical sequence $\{s_{k}\}\subset\left[0,\infty\right[$ 
such that $\sum_{k=1}^{\infty}s_{k}<\infty$, has as a practical disadvantage that there is no direct guidance as to how to select it when solving a specific problem.
Therefore, it is useful to construct error conditions for approximating
proximal subproblems that could be computable during the progress of the algorithm. 
Inexact versions of the PPA, which use relative error tolerance criteria of this kind, 
were developed in \cite{sol_sv_pp,sol_sv_hpe,sol_sv_unif}. 

To solve subproblems \eqref{sub_prob_b} and \eqref{sub_prob_a} inexactly, we will use the notion of approximate solutions of a proximal subproblem proposed in \cite{sol_sv_unif} by Solodov and Svaiter.

The general projective splitting framework for the sum of $m\geq2$ maximal monotone operators \cite{eck_sv_09} admits a relative error condition for
approximately evaluating resolvents. 
The criterion used in \cite{eck_sv_09} is a generalization for the case
of $m$ maximal monotone operators of the relative error tolerance of the \emph{hybrid proximal extragradient} method \cite{sol_sv_hpe}. 
We have preferred the framework developed in \cite{sol_sv_unif} since it yields a more flexible error tolerance criterion and evaluation of the $\epsilon$-enlargements of the operators.

We now present the notion of inexact solutions of a proximal subproblem introduced in \cite{sol_sv_unif}. Let $T:\mathbb{R}^{n}\rightrightarrows\mathbb{R}^{n}$ be a maximal monotone operator, 
$\lambda>0$ and $z\in\mathbb{R}^{n}$. Consider the \emph{proximal system}
\begin{equation}
\label{proximal_problem}
\left\lbrace\begin{array}{l}
w\in T(z'),\\
\lambda w + z' - z=0,
\end{array}\right.
\end{equation}
which is clearly equivalent to the proximal subproblem \eqref{eq:prox-sub}.

\begin{definition}
\label{Def:relative error}
\emph{
Given $\sigma\in\left[0,1\right[$, a triplet $(z',w,\epsilon)\in\mathbb{E}$ is called a $\sigma$\emph{-approximate solution} of (\ref{proximal_problem}) at $(\lambda,z)$, if 
\begin{equation}
\label{proximal_problem_error}
\begin{split}
& w\in T^{\epsilon}(z'),\\
& \norm{\lambda w+z'-z}^{2} +2\lambda\epsilon\leq\sigma\left(\norm{\lambda w}^{2}+\norm{z'-z}^{2}\right).
\end{split}
\end{equation}}
\end{definition}

We observe that if $(z',w)$ is the exact solution of (\ref{proximal_problem}) then, taking $\epsilon=0$, the triplet $(z',w,\epsilon)$ satisfies the approximation criterion 
(\ref{proximal_problem_error}) for all $\sigma\in\left[0,1\right[$. Conversely, if $\sigma=0$, only the exact solution of (\ref{proximal_problem}), with $\epsilon=0$, will satisfy (\ref{proximal_problem_error}).

The method that will be studied in this section is as follows.

\begin{algorithm}
\label{alg_inex}
Choose $(z_{0},w_{0})\in\R^n\times\R^n$, $\sigma\in\left[0,1\right[$ and $\overline{\rho}\in\left[0,1\right[$. Then, for $k=1,2,\ldots$
\begin{itemize}
\item[1.] Choose $\lambda_{k}$, $\mu_{k}>0$ and calculate $(x_{k},b_{k},\epsilon^x_{k})$ and $(y_{k},a_{k},\epsilon^y_{k})\in\mathbb{E}$ such that  
\begin{equation}
\label{error_criterion:B}
\begin{split}
&\,\,b_{k} \in B^{\epsilon^x_{k}}(x_{k}),\quad\qquad \lambda_{k}(b_{k}-w_{k-1}) = z_{k-1} -x_k + r^x_{k},\\
&\norm{r^x_k}^{2} + 2\lambda_{k}\epsilon^x_{k} \leq \sigma\left(\norm{x_{k}-z_{k-1}}^{2} + \norm{\lambda_{k}(b_{k}-w_{k-1})}^{2}\right),
\end{split}
\end{equation}
and
\begin{equation}
\label{error_criterion:A}
\begin{split}
& \,\, a_{k} \in A^{\epsilon^y_{k}}(y_{k}),\quad\qquad \mu_{k}(a_{k}+w_{k-1})=z_{k-1} - y_k + r^y_{k},\\
& \norm{r^y_ k}^{2} + 2\mu_{k}\epsilon^y_{k} \leq \sigma\left(\norm{y_{k}-z_{k-1}}^{2} + \norm{\mu_{k}(a_{k}+w_{k-1})}^{2}\right).
\end{split}
\end{equation}

\item[2.] If $\norm{a_{k}+b_{k}} + \norm{x_{k}-y_{k}}=0$ stop. Otherwise, set 
$$\gamma_{k}=\frac{\inpr{z_{k-1}-x_{k}}{b_{k}-w_{k-1}} + \inpr{z_{k-1}-y_{k}}{a_{k}+w_{k-1}} - \epsilon^x_{k} - \epsilon^y_{k}}{\norm{a_{k}+b_{k}}^{2} + \norm{x_{k}-y_{k}}^{2}}.$$ 

\item[3.] Choose a parameter $\rho_{k}\in[1-\overline{\rho},1+\overline{\rho}]$ and set
\begin{align*}
z_{k} = &\,\, z_{k-1} - \rho_{k}\gamma_{k}(a_{k}+b_{k}),\\
w_{k} = &\,\, w_{k-1} - \rho_{k}\gamma_{k}(x_{k}-y_{k}).
\end{align*}
\end{itemize}
\end{algorithm}

Note that for all iteration $k$, the triplet $(x_{k},b_{k},\epsilon^x_{k})$ calculated in step 1 of Algorithm \ref{alg_inex} is a $\sigma$-approximate solution of (\ref{proximal_problem}) at $(\lambda_{k},z_{k-1})$,
where $T=B-w_{k-1}$. Similarly, $(y_{k},a_{k},\epsilon^y_{k})$ is a $\sigma$-approximate solution of (\ref{proximal_problem}) (with $T=A+w_{k-1}$) at point $(\mu_{k},z_{k-1})$.
Observe also that taking $\sigma=0$ in Algorithm \ref{alg_inex} yields exactly the PSM with $\alpha_k=0$ for all integer $k\geq1$, since condition \eqref{ass_con} is satisfied.

Let us denote by $\phi_k$ the decomposable separator associated with the pair $(x_{k},b_{k},\epsilon^x_{k})$ and $(y_{k},a_{k},\epsilon^y_{k})$, for every integer $k\geq1$ (see Definition \ref{def_dec_sep_gen}).
It will be shown in the following lemma that if Algorithm \ref{alg_inex} stops in step 2 at iteration $k$, then it has found a point in $S_e(A,B)$. Otherwise we will have $\norm{\nabla\phi_{k}}>0$, which gives
$\phi_{k}(z_{k-1},w_{k-1})>0$. This clearly implies that Algorithm \ref{alg_inex} falls within the general framework of Algorithm \ref{alg_genr_proj}. 

\begin{lemma}
\label{Lem phi inex}
Let $\{(x_{k},b_{k},\epsilon^x_{k})\}$, $\{(y_{k},a_{k},\epsilon^y_{k})\}$, $\{(z_{k},w_{k})\}$, $\{\lambda_k\}$, $\{\mu_k\}$ and $\{\rho_k\}$ be the sequences generated by Algorithm $\ref{alg_inex}$, and 
$\{\phi_k\}$ be the sequence of decomposable separators associated with Algorithm $\ref{alg_inex}$. Then, for every integer $k\geq1$, we have
\begin{equation}
\label{phi_lower_est}
\phi_{k}(z_{k-1},w_{k-1})\geq\frac{1-\sigma}{4}\xi_{k}\left(\norm{a_{k}+b_{k}}^{2}+\norm{x_{k}-y_{k}}^{2}\right)\geq0,
\end{equation}
where 
\begin{equation}
\label{def_xi}
\xi_{k}:=\min\left\lbrace\lambda_{k},\frac{1}{\lambda_{k}},\mu_{k},\frac{1}{\mu_{k}}\right\rbrace.
\end{equation}
If $\norm{\nabla\phi_{k}}>0$, then it follows that $\phi_{k}(z_{k-1},w_{k-1})>0$. Furthermore, $\norm{\nabla\phi_{k}}=0$ if and only if $(x_{k},b_{k})=(y_{k},-a_{k})\in\exset{A}{B}$.
\end{lemma}

\begin{proof}
From the definition of $\phi_{k}$ and direct calculations it follows that
\begin{equation*}
\begin{split}
\phi_{k}(z_{k-1},w_{k-1}) = & \,\,\inpr{z_{k-1}-x_{k}}{b_{k}-w_{k-1}} + \inpr{z_{k-1}-y_{k}}{a_{k}+w_{k-1}} - \epsilon^x_{k} - \epsilon^y_{k}\\
= & \,\,\frac{1}{2\lambda_{k}}\left(\norm{z_{k-1}-x_{k}}^{2} + \norm{\lambda_{k}(b_{k}-w_{k-1})}^{2} - \norm{r^x_{k}}^{2} - 2\lambda_{k}\epsilon^x_{k}\right)\\
& \,\, + \frac{1}{2\mu_{k}}\left(\norm{z_{k-1}-y_{k}}^{2} + \norm{\mu_{k}(a_{k}+w_{k-1})}^{2} - \norm{r^y_{k}}^{2} - 2\mu_{k}\epsilon^y_{k}\right).\\
\end{split}
\end{equation*}
The identity above, together with the error criteria (\ref{error_criterion:B}) and (\ref{error_criterion:A}), implies 
\begin{equation}
\label{lower_est_phi}
\begin{split}
\phi_{k}(z_{k-1},w_{k-1})\geq &\, \frac{1-\sigma}{2\lambda_{k}}\left(\norm{z_{k-1}-x_{k}}^{2} + \norm{\lambda_{k}(b_{k}-w_{k-1})}^{2}\right)\\
&\, + \frac{1-\sigma}{2\mu_{k}}\left(\norm{z_{k-1}-y_{k}}^{2} + \norm{\mu_{k}(a_{k}+w_{k-1})}^{2}\right).
\end{split}
\end{equation}
If\,\, we \,\,interpret\,\, this\,\, last\,\, expression\,\, as\,\, a\,\, quadratic\,\, form\,\, applied \,\,to\,\, the\,\, $\mathbb{R}^{4}$\,\, vector\\
$(\norm{z_{k-1}-x_{k}},\norm{b_{k}-w_{k-1}},\norm{z_{k-1}-y_{k}},\norm{a_{k}+w_{k-1}})^{T}$, we obtain 
\begin{equation}
\label{matrix_4}
\begin{split}
\phi_{k}(z_{k-1},w_{k-1}) \geq &\,\frac{1-\sigma}{2}\left(\begin{array}{c}
\norm{z_{k-1}-x_{k}}\\
\norm{w_{k-1}-b_{k}}\\
\norm{z_{k-1}-y_{k}}\\
\norm{w_{k-1}+a_{k}}
\end{array}\right)^{T}
\left(\begin{array}{cccc}
\frac{1}{\lambda_{k}}& 0& 0& 0\\
0 &\lambda_{k}& 0 &0\\
0 &0& \frac{1}{\mu_{k}}& 0\\
0 &0 &0 &\mu_{k}
\end{array}\right)
\left(\begin{array}{c}
\norm{z_{k-1}-x_{k}}\\
\norm{w_{k-1}-b_{k}}\\
\norm{z_{k-1}-y_{k}}\\
\norm{w_{k-1}+a_{k}}
\end{array}\right)\\
\geq &\, \frac{1-\sigma}{2}\xi_{k}\left(\norm{z_{k-1}-x_{k}}^{2}+\norm{b_{k}-w_{k-1}}^{2}+\norm{z_{k-1}-y_{k}}^{2}+\norm{a_{k}+w_{k-1}}^{2}\right),
\end{split}
\end{equation}
where $\xi_{k}$, defined in (\ref{def_xi}), is the smallest eigenvalue of the matrix in (\ref{matrix_4}).

Now, we combine the second inequality in (\ref{matrix_4}) with relations 
$$\norm{z_{k-1}-x_{k}}^{2}+\norm{z_{k-1}-y_{k}}^{2}\geq\frac{1}{2}\norm{x_{k}-y_{k}}^{2},$$ $$\norm{b_{k}-w_{k-1}}^{2}+\norm{a_{k}+w_{k-1}}^{2}\geq\frac{1}{2}\norm{a_{k}+b_{k}}^{2};$$
to obtain the first inequality in (\ref{phi_lower_est}). Since $\xi_{k}>0$ and $\sigma\in\left[0,1\right[$, the second inequality in \eqref{phi_lower_est} follows directly. Furthermore, these last relations, together with equation (\ref{phi_lower_est}), clearly imply that $\phi_{k}(z_{k-1},w_{k-1})>0$ whenever $\norm{\nabla\phi_{k}}>0$.

To prove the last assertion of
the lemma we rewrite $\phi_{k}(z_{k-1},w_{k-1})$ as
\begin{equation*}
\phi_{k}(z_{k-1},w_{k-1}) = \inpr{a_{k}+b_{k}}{z_{k-1}-y_{k}} + \inpr{x_{k}-y_{k}}{w_{k-1}-b_{k}} - \epsilon^x_{k} - \epsilon^y_{k}.
\end{equation*}
Then, if $\norm{\nabla\phi_{k}}=0$, it follows that $x_{k}=y_{k}$, $b_{k}=-a_{k}$ and 
$$\phi_{k}(z_{k-1},w_{k-1})= - \epsilon^x_{k} - \epsilon^y_{k}.$$ 
From equation (\ref{phi_lower_est}), the equality above and the fact that $\epsilon^x_{k},\,\epsilon^y_{k}\geq0$, we obtain $\epsilon^x_{k}=\epsilon^y_{k}=0$. Hence,
$b_{k}\in B(x_{k})$, $a_{k}\in A(y_{k})$ and we conclude that $(x_{k},b_{k})=(y_{k},-a_{k})\in\exset{A}{B}$.
\end{proof}

For deriving complexity bounds for Algorithm \ref{alg_inex} we will assume, as was done in the preceding section,
that the method does not stop in a finite number of iterations. Thus, from now on we suppose that $\norm{\nabla\phi_{k}}>0$ for all integer $k\geq1$.

Next result establishes pointwise complexity bounds for Algorithm \ref{alg_inex}.

\begin{theorem}
\label{Teo ite comp inex}
Take $(z_0,w_0)\in\R^n\times\R^n$ and let  $\{(x_{k},b_{k},\epsilon^x_{k})\}$, $\{(y_{k},a_{k},\epsilon^y_{k})\}$, $\{\lambda_k\}$, $\{\mu_k\}$, $\{\gamma_k\}$ and $\{\rho_k\}$ be the sequences generated by Algorithm $\ref{alg_inex}$. 
Let $d_{0}$ be the distance of $(z_{0},w_{0})$ to the set $\exset{A}{B}$ and, for all integer $k\geq1$, define $\xi_k$ by \eqref{def_xi}. Then, for every integer $k\geq1$, we have 
\begin{equation}
\label{inclusions it parallel}
b_k\in B^{\epsilon^x_{k}}(x_k),\qquad\qquad a_k\in A^{\epsilon^y_{k}}(y_k),
\end{equation}
and there exists an index $1\leq i\leq k$ such that 
\begin{align*}
\norm{a_{i}+b_{i}}^{2} + \norm{x_{i}-y_{i}}^{2} \leq & \,\frac{16d_{0}^{2}}{(1-\sigma)^{2}(1-\overline{\rho})^{2}\xi_{i}\sum\limits_{j=1}^{k}\xi_{j}},\\
\epsilon^x_{i} + \epsilon^y_{i} \leq &\, \frac{4\sigma d_{0}^{2}}{(1-\sigma)^{2}(1-\overline{\rho})^{2}\sum\limits_{j=1}^{k}\xi_{j}}.
\end{align*}
\end{theorem}

\begin{proof}
The inclusions in \eqref{inclusions it parallel} are due to step 1 of Algorithm \ref{alg_inex}. Since $\gamma_{k}=\dfrac{\phi_{k}(z_{k-1},w_{k-1})}{\norm{\nabla\phi_{k}}^{2}}$, using (\ref{phi_lower_est}) we have 
\begin{equation}
\label{gamm lower est}
\gamma_{k} \geq \frac{1-\sigma}{4}\xi_{k}\quad \qquad\text{for } k=1,2,\dots.
\end{equation}
Thus, squaring both sides of (\ref{gamm lower est}) and multiplying by $\norm{\nabla\phi_{k}}^{2}$ we obtain
\begin{equation}
\label{lower est gm nb}
\gamma_{k}^{2}\norm{\nabla\phi_{k}}^{2} \geq \left(\frac{1-\sigma}{4}\right)^{2}\xi_{k}^{2}\norm{\nabla\phi_{k}}^{2}.
\end{equation}
Now, we observe that the error criteria (\ref{error_criterion:B}) and (\ref{error_criterion:A}) imply
\begin{equation*}
\epsilon^x_{k} \leq \frac{\sigma}{2\lambda_{k}}\left(\norm{z_{k-1}-x_{k}}^{2} + \norm{\lambda_{k}(b_{k}-w_{k-1})}^{2}\right)
\end{equation*}
and
\begin{equation*}
\epsilon^y_{k} \leq \frac{\sigma}{2\mu_{k}}\left(\norm{z_{k-1}-y_{k}}^{2} + \norm{\mu_{k}(a_{k}+w_{k-1})}^{2}\right),
\end{equation*}
respectively. 
Adding these two inequalities and combining with relation (\ref{lower_est_phi}) we obtain 
\begin{equation*}
\begin{split}
\epsilon^x_{k} + \epsilon^y_{k} \leq  \frac{\sigma}{1-\sigma}\phi_{k}(z_{k-1},w_{k-1})= \frac{\sigma}{1-\sigma}\gamma_{k}\norm{\nabla\phi_{k}}^{2}.
\end{split}
\end{equation*}
Multiplying the latter inequality by $\gamma_{k}$, using (\ref{gamm lower est}) and multiplying both sides of the resulting expression by $\dfrac{1-\sigma}{\sigma}$, we have
\begin{equation}
\label{upper_est:epsilon_k}
\frac{(1-\sigma)^{2}}{4\sigma}\xi_{k}(\epsilon^x_{k} + \epsilon^y_{k}) \leq \gamma_{k}^{2}\norm{\nabla\phi_{k}}^{2}\qquad\quad \textnormal{for }\, k=1,2,\ldots.
\end{equation}

Now, we define
\begin{equation*}
\psi_{k} := \max\left\lbrace\left(\frac{1-\sigma}{4}\right)^{2}\xi_{k}\norm{\nabla\phi_{k}}^{2},\frac{(1-\sigma)^{2}}{4\sigma}\left(\epsilon^x_{k}+\epsilon^y_{k}\right)\right\rbrace,
\end{equation*}
and combine (\ref{lower est gm nb}) with (\ref{upper_est:epsilon_k}) to obtain
\begin{equation*}
\xi_{k}\psi_{k} \leq \gamma_{k}^{2}\norm{\nabla\phi_{k}}^{2} \qquad\qquad \textnormal{for }\, k=1,2,\ldots.
\end{equation*}
Next, adding the inequality above from $j=1$ to $k$, using the assumption that $\rho_{k}\in[1-\overline{\rho},1+\overline{\rho}]$ for all integer $k\geq1$ and the first inequality in (\ref{dist_est}), we have
\begin{equation*}
\sum_{j=1}^{k}\xi_{j}\psi_{j}  \leq \frac{d_{0}^{2}}{(1-\overline{\rho})^{2}},
\end{equation*}
 and consequently
\begin{equation*}
\left(\min_{j=1,\ldots,k}\{\psi_{j}\}\right)\sum_{j=1}^{k}\xi_{j} \leq \frac{d_{0}^{2}}{(1-\overline{\rho})^{2}}.
\end{equation*}
The theorem now follows from this last inequality and the definition of $\psi_{k}$.
\end{proof}

If \,$\{(x_{k},b_{k},\epsilon^x_{k})\}$, $\{(y_{k},a_{k},\epsilon^y_{k})\}$, $\{\gamma_k\}$ and $\{\rho_k\}$ are the sequences generated by Algorithm \ref{alg_inex},
we consider their associated sequences of ergodic iterates $\{(\overline{x}_{k},\overline{b}_{k},\overline{\epsilon}^x_{k})\}$ and $\{(\overline{y}_{k},\overline{a}_{k},\overline{\epsilon}^y_{k})\}$, defined as in (\ref{def_xy_erg}),
(\ref{def_res_erg_x}) and (\ref{def_res_erg_y}). Since Algorithm \ref{alg_inex} is a special instance of Algorithm \ref{alg_genr_proj}, the results of subsection \ref{ergodic iterates} hold for its ergodic sequences. 
Thus, combining Theorem \ref{Teo erg bound} and Lemma \ref{Lem phi inex} we can derive ergodic complexity estimates for the method.

\begin{theorem}
\label{Teo_compl_erg_inex}
Let $\{(x_{k},b_{k},\epsilon^x_{k})\}$, $\{(y_{k},a_{k},\epsilon^y_{k})\}$, $\{\gamma_k\}$ and $\{\rho_k\}$ be the sequences generated by Algorithm $\ref{alg_inex}$. Let 
$\{(\overline{x}_{k},\overline{b}_{k},\overline{\epsilon}^x_{k})\}$ and $\{(\overline{y}_{k},\overline{a}_{k},\overline{\epsilon}^y_{k})\}$ be the sequences of ergodic iterates associated with 
Algorithm $\ref{alg_inex}$, defined as in \eqref{def_xy_erg}-\eqref{def_res_erg_y}, and consider $\xi_k$ given by \eqref{def_xi}. Then, for all integer $k\geq1$, we have
\begin{align}
\label{inclusion inex}
\overline{b}_k\in B^{\overline{\epsilon}^x_{k}}(\overline{x}_k),\qquad\qquad \overline{a}_k\in A^{\overline{\epsilon}^y_{k}}(\overline{y}_k)
\end{align}
and
\begin{align}
\label{bounds erg inex}
&\norm{\overline{a}_{k}+\overline{b}_{k}} \leq \frac{2d_{0}}{\Gamma_{k}}, &\norm{\overline{x}_{k}-\overline{y}_{k}} \leq  \frac{2d_{0}}{\Gamma_{k}},& &\overline{\epsilon}^x_{k} + \overline{\epsilon}^y_{k} \leq \frac{d_{0}^{2}(\varphi_{k}+4)}{\Gamma_{k}},
\end{align}
where $d_{0}$ is the distance of $(z_{0},w_{0})$ to $\exset{A}{B}$ and 
$$\varphi_{k}:=\left(\frac{2}{1-\sigma}\right)\max_{j=1,\ldots,k}\left\lbrace\frac{1}{\xi_{j}(2-\rho_{j})\Gamma_{k}}\right\rbrace.$$
\end{theorem}

\begin{proof}
The inclusions in \eqref{inclusion inex} follow from Lemma \ref{Lem_residual_erg}. The first two bounds in (\ref{bounds erg inex}) are due to \eqref{compl_erg_ax} in Theorem \ref{Teo erg bound}. 

Now, we note that the second inequality in (\ref{matrix_4}) implies
\begin{equation*}
\left(\norm{z_{j-1}-y_{j}}^{2}+\norm{b_{j}-w_{j-1}}^{2}\right)\frac{1-\sigma}{2}\xi_{j}\leq \phi_{j}(z_{j-1},w_{j-1})\qquad \text{for } j=1,2,\dots.
\end{equation*}
The relation above, together with the definition of $\gamma_{j}$, yields
\begin{equation*}
\norm{z_{j-1}-y_{j}}^{2}+\norm{b_{j}-w_{j-1}}^{2}\leq \frac{2}{(1-\sigma)\xi_j}\gamma_j\norm{\nabla\phi_j}^2\qquad \text{for } j=1,2,\dots.
\end{equation*}
Multiplying the above inequality by $\dfrac{1}{\Gamma_k}\rho_j\gamma_j$ and adding from $j=1$ to $k$, we obtain
\begin{equation*}
\begin{split}
\frac{1}{\Gamma_{k}}\sum_{j=1}^{k}\rho_{j}\gamma_{j}\norm{(y_{j},b_{j})-(z_{j-1},w_{j-1})}^{2} \leq &\, \frac{1}{\Gamma_{k}}\sum_{j=1}^{k}\frac{2}{(1-\sigma)\xi_{j}}\rho_{j}\gamma_{j}^{2}\norm{\nabla\phi_{j}}^{2}\\
= &\, \frac{1}{\Gamma_{k}}\sum_{j=1}^{k}\frac{2}{(1-\sigma)\xi_{j}(2-\rho_j)}\rho_{j}(2-\rho_j)\gamma_{j}^{2}\norm{\nabla\phi_{j}}^{2}\\ 
\leq &\, \varphi_{k}\sum_{j=1}^{k}\rho_{j}(2-\rho_{j})\gamma_{j}^{2}\norm{\nabla\phi_{j}}^{2}\\
\leq & \,\varphi_{k}d_{0}^{2},
\end{split}
\end{equation*}
where the second and the third inequalities are due to the definition of $\varphi_{k}$ and the first bound in (\ref{dist_est}), respectively. 
Substituting equation above into (\ref{compl_erg_ep}) we obtain the last bound in \eqref{bounds erg inex}.
\end{proof}

Theorems \ref{Teo ite comp inex} and \ref{Teo_compl_erg_inex} provide general complexity results for Algorithm \ref{alg_inex}.
Observe that the derived bounds are expressed in terms of $\xi_{k}$ and $\Gamma_{k}$. Next result, which is a direct consequence of these theorems, presents iteration-complexity
bounds for Algorithm \ref{alg_inex} to obtain $(\delta,\epsilon)$-approximate solutions of problem \eqref{problem}.

 \begin{theorem}
 \label{Teo approx solution}
 Assume the hypotheses of Theorem $\ref{Teo_compl_erg_inex}$. Assume also condition \emph{(A.1)} and
 define $\xi:=\min\left\lbrace\underline{\lambda},\dfrac{1}{\overline{\lambda}}\right\rbrace$. Then, for all $\delta,\,\epsilon>0$, the following statements hold.
 \begin{itemize}
 \item[(a)] There exists an index 
 \begin{equation*}
 i = \mathcal{O}\left(\max\left\lbrace\frac{d_{0}^{2}}{\xi^{2}\delta^{2}},\frac{d_{0}^{2}}{\xi\epsilon}\right\rbrace\right)
 \end{equation*}
such that the iterate $(x_i,y_i)$ is a $(\delta,\epsilon)$-solution of problem \eqref{problem}. 
 \item[(b)] There exists an index
 \begin{equation*}
 k_{0} = \mathcal{O}\left(\max\left\lbrace \frac{d_{0}}{\xi\delta},\frac{d_{0}^{2}}{\xi\epsilon} \right\rbrace\right)
 \end{equation*}
 such that, for any $k \geq k_{0}$, the ergodic iterate $(\overline{x}_{k},\overline{y}_k)$ is a $(\delta,\epsilon)$-solution of problem \eqref{problem}. 
 \end{itemize}
 \end{theorem}

\begin{proof}
We first note that assumption (A.1) implies
\begin{equation}
\label{eq:xi_k}
\xi_{j}\geq\xi\qquad\qquad \text{for }\,\,j=1,2,\dots.
\end{equation}
Now, we combine the definition of $\Gamma_k$ in \eqref{def_xy_erg} with \eqref{gamm lower est} and \eqref{eq:xi_k} to obtain
 \begin{equation*}
 \Gamma_{k} \geq k(1-\overline{\rho})\xi\frac{(1-\sigma)}{4}.
 \end{equation*}
Furthermore, the inequality above, together with \eqref{eq:xi_k} and the definition of $\varphi_{k}$, yields 
 \begin{equation*}
 \varphi_{k} \leq \frac{8}{(1-\overline{\rho})^2(1-\sigma)^2\xi^2}.
 \end{equation*}
We conclude the proof combining Theorems \ref{Teo ite comp inex} and \ref{Teo_compl_erg_inex} with these three relations above. 
\end{proof}

\section{A Sequential Inexact Case}
\label{sequential}
In this section, we propose an inexact variant of a sequential case of the PSM and study its iteration-complexity. We observe that, unless $\alpha_{k}=0$, subproblems (\ref{sub_prob_b}) and (\ref{sub_prob_a}) cannot
be solved in parallel. For example, if we specialize the PSM by setting $\alpha_{k}=1$ for all $k$, we have to perform on each iteration the following steps
\begin{align*}
\lambda_{k}b_{k} + x_{k} &= z_{k-1} + \lambda_{k}w_{k-1},&& b_{k}\in B(x_{k}),\\
\mu_{k}a_{k} + y_{k} &= x_{k} - \mu_{k}w_{k-1},&& a_{k}\in A(y_{k}).
\end{align*}
Therefore, the first problem above must be solved after the second one; these steps cannot be performed simultaneously like the proximal subproblems of Algorithm \ref{alg_inex}. However,
this choice of $\alpha_{k}$ could be an advantage since the second subproblem uses more recent information, that is $x_{k}$ instead of $z_{k-1}$. 

In this section, we are assuming that the resolvent mappings of operator $B$ are easy to evaluate, but this is not the case for the proximal mappings associated with operator $A$. Such situations are typical in practice even in the case of convex optimization. Indeed, if $A=\partial f$ and $B=\partial g$ are the subdifferential operators of functions $f$ and $g$, where $f$ and $g$ are proper, convex and lower semicontinuous, then the solutions of the MIP \eqref{problem} are minimizers of the sum $f+g$. In this case, 
in order to evaluate 
the resolvent mapping $(I+\lambda A)^{-1}=(I+\lambda\partial f)^{-1}$, it is necessary to solve a strongly convex minimization problem and, if $f$ has a complicated algebraic expression, such problem could be hard to solve exactly. Therefore, it is desirable to admit inexact solutions of the proximal subproblems associated with this operator.

With these assumptions, we propose the following modification of the specific case of the PSM where $\alpha_k=1$ for all iteration $k$. Specifically, in Algorithm \ref{alg_seq_inex} below we allow the solution of the second proximal subproblem to be approximated, provided that the approximate solution satisfies the relative error condition of Definition \ref{Def:relative error}.

\begin{algorithm}
\label{alg_seq_inex}
Choose $(z_{0},w_{0})\in\R^n\times\R^n$, $\sigma\in\left[0,1/2\right[$ and $\overline{\rho}\in\left[0,1\right[$. Then, for $k=1,2,\dots$
\begin{itemize}
\item[1.] Choose $\lambda_{k}>0$ and calculate $(x_{k},b_{k})\in\R^n\times\R^n$ and $(y_{k},a_{k},\epsilon^y_{k})\in\mathbb{E}$ such that 
\begin{equation}
\label{subp_seq:B}
\lambda_{k}b_{k}+x_{k}=z_{k-1} + \lambda_{k}w_{k-1},\qquad\qquad b_{k}\in B(x_{k}),
\end{equation}
and
\begin{align}
\label{inclus_seq:A}
&\,\, \lambda_{k}a_{k} + y_{k}= x_{k}-\lambda_{k}w_{k-1} + r_{k}, \quad\qquad a_{k} \in A^{\epsilon^y_{k}}(y_{k}),\\
\label{error_criterion_seq:A}
&\norm{r_{k}}^{2} + 2\lambda_{k}\epsilon^y_{k}\, \leq \,\sigma\left(\norm{y_{k}-x_{k}}^{2} + \norm{\lambda_{k}(a_{k}+w_{k-1})}^{2}\right).
\end{align}

\item[2.] If $\norm{a_{k}+b_{k}} + \norm{x_{k}-y_{k}}=0$ stop. Otherwise, set 
$$\gamma_{k}=\frac{\inpr{z_{k-1}-x_{k}}{b_{k}-w_{k-1}} + \inpr{z_{k-1}-y_{k}}{a_{k}+w_{k-1}} - \epsilon^y_{k}}{\norm{a_{k}+b_{k}}^{2} + \norm{x_{k}-y_{k}}^{2}}.$$ 

\item[3.] Choose a parameter $\rho_{k}\in[1-\overline{\rho},1+\overline{\rho}]$ and set
\begin{align*}
z_{k} = &\,\, z_{k-1} - \rho_{k}\gamma_{k}(a_{k}+b_{k}),\\
w_{k} = &\,\, w_{k-1} - \rho_{k}\gamma_{k}(x_{k}-y_{k}).
\end{align*}
\end{itemize}
\end{algorithm}

We note that the maximum tolerance for the relative error in the resolution of (\ref{inclus_seq:A})-(\ref{error_criterion_seq:A}) is $1/2$, instead of $1$ as in Algorithm \ref{alg_inex}.
We also note that the proximal parameter in step 1 of Algorithm \ref{alg_seq_inex} is not allowed to change from one subproblem to another within an iteration.

For every integer $k\geq1$ denote by $\phi_k$ the decomposable separator associated with the triplets $(x_{k},b_{k},0)$ and $(y_{k},a_{k},\epsilon^y_{k})$, calculated in step 1 of Algorithm \ref{alg_seq_inex} (see Definition \ref{def_dec_sep_gen}).
It is thus clear that if $\phi_{k}(z_{k-1},w_{k-1})>0$ for all integer $k\geq1$, then Algorithm \ref{alg_seq_inex} is an instance of the general scheme presented in section \ref{sec:general framework}.

The following lemma implies that Algorithm \ref{alg_seq_inex} stops in step 2 when it has found a point in the extended solution set $\exset{A}{B}$.

\begin{lemma}
\label{Lem phi seq}
Let $\{(x_{k},b_{k})\}$, $\{(y_{k},a_{k},\epsilon^y_{k})\}$, $\{(z_{k},w_{k})\}$, $\{\lambda_k\}$ and $\{\rho_k\}$ be the sequences generated by Algorithm $\ref{alg_seq_inex}$, 
and $\{\phi_k\}$ be the sequence of decomposable separators associated with Algorithm $\ref{alg_seq_inex}$. Then, for all integer $k\geq1$, we have
\begin{equation}
\label{phi_lower_est_seq}
\phi_{k}(z_{k-1},w_{k-1})\geq\frac{1-2\sigma}{2}\tau_{k}\left(\norm{a_{k}+b_{k}}^{2}+\norm{x_{k}-y_{k}}^{2}\right)\geq0,
\end{equation}
where 
\begin{equation}
\label{def_tau}
\tau_{k}:=\min\left\lbrace\lambda_{k},\frac{1}{\lambda_{k}}\right\rbrace.
\end{equation}
If $\norm{\nabla\phi_{k}}>0$, then it follows that $\phi_{k}(z_{k-1},w_{k-1})>0$. Furthermore, $\norm{\nabla\phi_{k}}=0$ if and only if $(x_{k},b_{k})=(y_{k},-a_{k})\in\exset{A}{B}$.
\end{lemma}

\begin{proof}
Since 
\begin{equation*}
\phi_{k}(z_{k-1},w_{k-1})=\inpr{z_{k-1}-x_{k}}{b_{k}-w_{k-1}}  + \inpr{z_{k-1}-y_{k}}{a_{k}+w_{k-1}} - \epsilon^y_{k},
\end{equation*}
adding and subtracting $\inpr{x_{k}}{a_{k}+w_{k-1}}$ on the right-hand side of this equation and regrouping the terms, we obtain
\begin{equation}
\label{identity_phi_seq}
\begin{split}
\phi_{k}(z_{k-1},w_{k-1})=&\,\inpr{z_{k-1}-x_{k}}{b_{k}+a_{k}}  + \inpr{x_{k}-y_{k}}{a_{k}+w_{k-1}} - \epsilon^y_{k}\\
= &\, \lambda_{k}\inpr{b_{k}-w_{k-1}}{b_{k}+a_{k}} + \frac{1}{2\lambda_{k}}\left[\norm{x_{k}-y_{k}}^{2}+\norm{\lambda_{k}(a_{k}+w_{k-1})}^{2}\right]\\
& \,- \frac{1}{2\lambda_{k}}\left[\norm{r_{k}}^{2} + 2\lambda_{k}\epsilon^y_{k}\right],
\end{split}
\end{equation}
where we have used in the last equality the identity in (\ref{subp_seq:B}) and $r_{k}$ is given in (\ref{inclus_seq:A}). We observe that 
\begin{equation*}
\lambda_{k}\inpr{b_{k}-w_{k-1}}{b_{k}+a_{k}} = \frac{\lambda_{k}}{2}\left[\norm{b_{k}-w_{k-1}}^{2}+\norm{b_{k}+a_{k}}^{2}-\norm{a_{k}+w_{k-1}}^{2}\right].
\end{equation*}
Hence, combining equality above with \eqref{identity_phi_seq} and the error criterion (\ref{error_criterion_seq:A}) we have
\begin{align*}
\phi_{k}(z_{k-1},w_{k-1})\geq \frac{\lambda_{k}}{2}\norm{b_{k}-w_{k-1}}^{2} + \frac{\lambda_{k}}{2}\norm{a_{k}+b_{k}}^{2} + \frac{1-\sigma}{2\lambda_{k}}\norm{x_{k}-y_{k}}^{2} - \frac{\sigma\lambda_{k}}{2}\norm{a_{k}+w_{k-1}}^{2}.
\end{align*}
Since $\norm{a_{k}+w_{k-1}}^{2}\leq2\norm{a_{k}+b_{k}}^{2}+2\norm{b_{k}-w_{k-1}}^{2}$, we deduce that
\begin{equation}
\label{phi_seq_est}
\phi_{k}(z_{k-1},w_{k-1})\geq \frac{\lambda_{k}(1-2\sigma)}{2}\norm{b_{k}-w_{k-1}}^{2} + \frac{\lambda_{k}(1-2\sigma)}{2}\norm{a_{k}+b_{k}}^{2} + \frac{1-\sigma}{2\lambda_{k}}\norm{x_{k}-y_{k}}^{2}.
\end{equation}
The inequalities in (\ref{phi_lower_est_seq}) now follow from the relation above, the definition of $\tau_{k}$ and noting that $1-\sigma\geq1-2\sigma>0$.

The claim that $\norm{\nabla\phi_{k}}>0$ implies $\phi_{k}(z_{k-1},w_{k-1})>0$ is obtained as a direct consequence of (\ref{phi_lower_est_seq}). To prove the remaining assertion of the lemma we observe
that if $\norm{\nabla\phi_{k}}=0$, then $x_{k}=y_{k}$, $b_{k}=-a_{k}$ and it follows from \eqref{phi_lower_est_seq}, the first equality in (\ref{identity_phi_seq}) and the fact that $\epsilon^y_{k}\in\R_+$, that $\epsilon^y_{k}=0$.
Thus, we have $(x_{k},b_{k})\in\exset{A}{B}$.
\end{proof}

From now on we assume that Algorithm \ref{alg_seq_inex} generates infinite sequences $\{x_k\}$ and $\{y_k\}$, which is equivalent to $\norm{\nabla\phi_{k}}>0$ for every integer $k\geq1$. 

We are now ready to establish pointwise iteration-complexity bounds for Algorithm \ref{alg_seq_inex}.
The theorem below will be proven in much the same way as Theorem \ref{Teo ite comp inex}, using Lemma \ref{Lem phi seq} instead of Lemma \ref{Lem phi inex}. 

\begin{theorem}
\label{Teo it compl seq}
Take $(z_0,w_0)\in\R^n\times\R^n$ and let  $\{(x_{k},b_{k})\}$, $\{(y_{k},a_{k},\epsilon^y_{k})\}$, $\{\lambda_k\}$, $\{\gamma_k\}$ and $\{\rho_k\}$ be the sequences generated by Algorithm $\ref{alg_seq_inex}$. 
Let $d_{0}$ be the distance of $(z_{0},w_{0})$ to $\exset{A}{B}$ and, for all integer $k\geq1$, let $\tau_k$ be given by \eqref{def_tau}. Then, for every integer $k\geq1$, we have
\begin{equation}
\label{inclusions it sequential}
b_k\in B(x_k),\qquad\qquad a_k\in A^{\epsilon^y_{k}}(y_k),
\end{equation}
and there exists an index $1\leq i\leq k$ such that 
\begin{align*}
\norm{a_{i}+b_{i}}^{2} + \norm{x_{i}-y_{i}}^{2} \leq &\, \frac{4d_{0}^{2}}{(1-2\sigma)^{2}(1-\overline{\rho})^{2}\tau_{i}\sum\limits_{j=1}^{k}\tau_{j}},\\
\epsilon^y_{i} \leq &\, \frac{4\sigma d_{0}^{2}}{(1-2\sigma)^{2}(1-\overline{\rho})^{2}\sum\limits_{j=1}^{k}\tau_{j}}.
\end{align*}
\end{theorem}

\begin{proof}
The inclusions in \eqref{inclusions it sequential} are due to step 1 of Algorithm \ref{alg_seq_inex}. It follows from the definition of $\gamma_{k}$ and inequality (\ref{phi_lower_est_seq}) that
\begin{equation}
\label{est_gamm_seq}
\gamma_{k}\geq \left(\frac{1-2\sigma}{2}\right)\tau_{k}\qquad\quad \text{for }\,k=1,2,\dots.
\end{equation}
Squaring both sides of the above inequality and multiplying by $\norm{\nabla\phi_{k}}^{2}$ we obtain
\begin{equation}
\label{eq:nabla_1}
\gamma_{k}^{2}\norm{\nabla\phi_{k}}^{2}\geq \left(\frac{1-2\sigma}{2}\right)^{2}\tau_{k}^{2}\norm{\nabla\phi_{k}}^{2},\qquad\quad \text{for }\,k=1,2,\dots.
\end{equation}
Now, we note that the error criterion (\ref{error_criterion_seq:A}) implies
\begin{equation*}
\epsilon^y_{k}\leq\frac{\sigma}{2\lambda_{k}}\left[\norm{x_{k}-y_{k}}^{2}+\norm{\lambda_{k}(a_{k}+b_{k})}^{2}\right].
\end{equation*}
Consequently, we have
\begin{equation*}
\epsilon^y_{k}\leq\frac{\sigma}{2\lambda_{k}}\norm{x_{k}-y_{k}}^{2}+\sigma\lambda_{k}\norm{a_{k}+w_{k-1}}^{2} +\sigma\lambda_{k}\norm{b_{k}-w_{k-1}}^{2}.
\end{equation*}
The above inequality, together with (\ref{phi_seq_est}), yields
\begin{equation*}
\epsilon^y_{k} \leq \frac{2\sigma}{1-2\sigma}\phi_{k}(z_{k-1},w_{k-1}).
\end{equation*}
Next, multiplying the above relation by $\gamma_{k}$ and combining with (\ref{est_gamm_seq}), after some manipulations, we obtain
\begin{equation}
\label{eq:nabla_2}
\frac{(1-2\sigma)^{2}}{4\sigma}\tau_{k}\epsilon^y_{k}\leq \gamma_{k}^{2}\norm{\nabla\phi_{k}}^{2}.
\end{equation}
Finally, defining
\begin{equation*}
\psi_{k}:=\max\left\lbrace\frac{(1-2\sigma)^2}{4}\tau_{k}\norm{\nabla\phi_{k}}^{2},\frac{(1-2\sigma)^{2}}{4\sigma}\epsilon^y_{k}\right\rbrace
\end{equation*}
and using \eqref{eq:nabla_1} and \eqref{eq:nabla_2}, we can conclude the proof proceeding analogously to the proof of Theorem \ref{Teo ite comp inex}.
\end{proof}

The following theorem presents complexity estimates in the ergodic sense for Algorithm \ref{alg_seq_inex}.

\begin{theorem}
\label{Teo_compl_erg_seq}
Let $\{(x_{k},b_{k})\}$, $\{(y_{k},a_{k},\epsilon^y_{k})\}$, $\{\gamma_k\}$ and $\{\rho_k\}$ be the sequences generated by Algorithm $\ref{alg_seq_inex}$. Let 
$\{(\overline{x}_{k},\overline{b}_{k},\overline{\epsilon}^x_{k})\}$ and $\{(\overline{y}_{k},\overline{a}_{k},\overline{\epsilon}^y_{k})\}$ be the associated sequences of ergodic iterates, defined as in \eqref{def_xy_erg}-\eqref{def_res_erg_y}, and consider $\tau_k$ given by \eqref{def_tau}. Then, for all integer $k\geq1$, we have
\begin{equation}
\label{inclusion erg seq}
\overline{b}_k\in B^{\overline{\epsilon}^x_{k}}(\overline{x}_k),\qquad\qquad \overline{a}_k\in A^{\overline{\epsilon}^y_{k}}(\overline{y}_k),
\end{equation}
and
\begin{align}
\label{bounds erg seq}
&\norm{\overline{a}_{k}+\overline{b}_{k}} \leq \frac{2d_{0}}{\Gamma_{k}}, &\norm{\overline{x}_{k}-\overline{y}_{k}} \leq  \frac{2d_{0}}{\Gamma_{k}},& &\overline{\epsilon}^x_{k} + \overline{\epsilon}^y_{k} \leq \frac{d_{0}^{2}(\vartheta_{k}+4)}{\Gamma_{k}},
\end{align}
where $d_{0}$ is the distance of $(z_{0},w_{0})$ to $\exset{A}{B}$ and 
$$\vartheta_{k}:=\max_{j=1,\ldots,k}\left\lbrace\frac{8}{\tau_{j}(1-2\sigma)(2-\rho_{j})\Gamma_{k}}\right\rbrace.$$
\end{theorem}

\begin{proof}
Since Algorithm \ref{alg_seq_inex} is an instance of Algorithm \ref{alg_genr_proj}, Lemma \ref{Lem_residual_erg} and Theorem \ref{Teo erg bound} apply, therefore the inclusions in \eqref{inclusion erg seq} and the first two inequalities in (\ref{bounds erg seq}) follow.

We derive now an estimate for the sum on the right-hand side of (\ref{compl_erg_ep}). We note that (\ref{phi_seq_est}) implies
\begin{equation}
\label{ineq phi seq}
\phi_{j}(z_{j-1},w_{j-1})\geq \lambda_{j}\left(\frac{1-2\sigma}{2}\right)\norm{b_{j}-w_{j-1}}^{2}
\end{equation}
for all integer $j\geq1$. We also note that 
\begin{equation*}
z_{j-1}-y_{j}=z_{j-1}-x_{j}+x_{j}-y_{j}=\lambda_{j}(b_{j}-w_{j-1})+x_{j}-y_{j},
\end{equation*}
where the last identity is due to the equality in (\ref{subp_seq:B}). This last expression and the triangle inequality for norms yield
\begin{equation*}
 \norm{z_{j-1}-y_{j}} \leq \lambda_{j}\norm{b_{j}-w_{j-1}}+\norm{x_{j}-y_{j}}.
\end{equation*}
Moreover, squaring both sides of the inequality above and making some manipulations, we obtain
\begin{equation}
\label{eq:zybw}
\begin{split}
\frac{1}{2\lambda_{j}}\norm{z_{j-1}-y_{j}}^{2}&\leq \lambda_{j}\norm{b_{j}-w_{j-1}}^{2}+\frac{1}{\lambda_{j}}\norm{x_{j}-y_{j}}^{2}\\&\leq \frac{2}{1-2\sigma}\phi_{j}(z_{j-1},w_{j-1}),
\end{split}
\end{equation}
where the last inequality follows from \eqref{phi_seq_est}. Now, adding \eqref{ineq phi seq} and \eqref{eq:zybw} we have
\begin{equation*}
\frac{1}{2\lambda_j}\norm{z_{j-1}-y_j}^2+\lambda_j\norm{b_j-w_{j-1}}^2\leq  \frac{4}{1-2\sigma}\phi_{j}(z_{j-1},w_{j-1}).
\end{equation*}
The above relation, together with the definitions of $\gamma_j$ and $\tau_j$, implies
\begin{equation*}
\norm{b_{j}-w_{j-1}}^{2} + \norm{z_{j-1}-y_{j}}^{2} \leq \frac{8}{(1-2\sigma)\tau_{j}}\gamma_{j}\norm{\nabla\phi_{j}}^{2}.
\end{equation*}
Multiplying both sides of the above inequality by $\dfrac{1}{\Gamma_k}\rho_{j}\gamma_{j}$ and adding from $j=1$ to $k$, we obtain the desired estimate, i.e.
\begin{align*}
 \dfrac{1}{\Gamma_k}\sum_{j=1}^k\rho_j\gamma_j\left[\norm{b_{j}-w_{j-1}}^{2} + \norm{z_{j-1}-y_{j}}^{2}\right] &\leq \dfrac{1}{\Gamma_k}\sum_{j=1}^k\frac{8}{(1-2\sigma)\tau_{j}}\rho_j\gamma_{j}^2\norm{\nabla\phi_{j}}^{2}\\
 &= \dfrac{1}{\Gamma_k}\sum_{j=1}^k\frac{8}{(1-2\sigma)\tau_{j}(2-\rho_j)}\rho_j(2-\rho_j)\gamma_{j}^2\norm{\nabla\phi_{j}}^{2}\\
 & \leq \vartheta_k\sum_{j=1}^k\rho_j(2-\rho_j)\gamma_{j}^2\norm{\nabla\phi_{j}}^{2}\\
 & \leq\vartheta_kd_0^2,
\end{align*}
where the second and the third inequalities above follow from the definition of $\vartheta_k$ and \eqref{dist_est}, respectively.
The proof of the last bound in (\ref{bounds erg seq}) now follows combining the above relation with \eqref{compl_erg_ep}.
\end{proof}

Next result provides complexity bounds for Algorithm \ref{alg_seq_inex} to find a $(\delta,\epsilon)$-approximate solution of problem 
(\ref{problem}).
It may be proven in much the same way as Theorem \ref{Teo approx solution} and for the sake of brevity we omit the proof here. 

\begin{theorem}
Assume the hypotheses of Theorem $\ref{Teo_compl_erg_seq}$. Suppose also that there exist $\overline{\lambda}$ and $\underline{\lambda}$ such that $\overline{\lambda}\geq\underline{\lambda}>0$ and
$\lambda_{k}\in[\underline{\lambda},\overline{\lambda}]$, for all integer $k\geq1$, and define $\tau:=\min\left\lbrace\underline{\lambda},\dfrac{1}{\overline{\lambda}}\right\rbrace$. Then, for every $\delta,\,\epsilon>0$, the following claims hold.
\begin{itemize}
\item[(a)] There exists an index 
\begin{equation*}
i = \mathcal{O}\left(\max\left\lbrace\frac{d_{0}^{2}}{\tau^{2}\delta^{2}},\frac{d_{0}^{2}}{\tau\epsilon}\right\rbrace\right)
\end{equation*}
such that the point $(x_{i},y_{i})$ calculated by Algorithm $\ref{alg_seq_inex}$ is a $(\delta,\epsilon)$-approximate solution of problem \eqref{problem}.
\item[(b)] There exists an index
\begin{equation*}
k_{0} = \mathcal{O}\left(\max\left\lbrace \frac{d_{0}}{\tau\delta},\frac{d_{0}^{2}}{\tau\epsilon} \right\rbrace\right)
\end{equation*}
such that, for any $k \geq k_{0}$, the ergodic iterate $(\overline{x}_{k},\overline{y}_{k})$ is a $(\delta,\epsilon)$-approximate solution of problem \eqref{problem}.
\end{itemize}
\end{theorem}

\section{Conclusions}
We introduced a general projective splitting scheme for solving monotone inclusion problems given by the sum of two maximal monotone operators, which generalizes the family of projective splitting methods (PSM) proposed by Eckstein and Svaiter. Using this general framework we analyzed the iteration-complexity of the family of PSM and, as a consequence, we obtained the iteration-complexity of the two-operator case of the Spingarn partial inverse method. We introduced two inexact variants of two special cases of the family of PSM, which allow the resolvent mappings to be solved inexactly. We also proved the iteration-complexity for the above-mentioned methods.

\vspace{2mm}
\noindent {\bf Acknowledgments} \,\, {\small This work is part of the author's Ph.D. thesis, written under the supervision of Benar Fux Svaiter at IMPA, and supported by CAPES and FAPERJ.}

\bibliographystyle{spmpsci_unsrt}
\bibliography{spingarn-complexity.bbl}

\end{document}